\documentclass[a4paper,11pt]{article}
\usepackage{amsmath,amssymb,amsthm}
\usepackage{cases}
\usepackage{array}   
\usepackage{a4wide}
\usepackage{graphicx}
\graphicspath{{./}}
\usepackage[hang,small,bf]{caption}
\usepackage[subrefformat=parens]{subcaption}
\usepackage{here}
\usepackage{placeins}
\usepackage{comment}
\usepackage{ulem}
\usepackage{xcolor}
\usepackage{tikz}
\usetikzlibrary{arrows.meta,calc,positioning}

\usepackage{mathrsfs}
\usepackage{stmaryrd}

\usepackage[pdfencoding=auto, psdextra]{hyperref}
\hypersetup{
    plainpages=false,
    unicode=true,          
    pdftoolbar=true,        
    pdfmenubar=true,        
    pdffitwindow=true,      
    pdftitle={A Structure-Preserving Method of Fundamental Solutions for the Multi-Phase Mullins--Sekerka Flow},
    pdfauthor={Tokuhiro Eto},
    pdfsubject={Multi-phase Mullins-Sekerka flow; charge simulation method},
    pdfkeywords={multi-phase Mullins-Sekerka flow; charge simulation method; method of fundamental solutions; triple junction; structure-preserving discretization},
    pdfnewwindow=true,      
    colorlinks=true,       
    linkcolor=blue,          
    citecolor=blue,        
    filecolor=blue,      
    urlcolor=magenta,          
    urlbordercolor=0 1 1
}

\numberwithin{equation}{section}


\newtheorem{prop}{Proposition}[section]

\newtheorem{lem}{Lemma}[section]
\newtheorem{rem}{Remark}[section]

\newtheorem{asm}{Assumption}[section]

\newcommand{\Section}[1]{Section~\ref{#1}}

\newcommand{\Remark}[1]{Remark~\ref{#1}}

\newcommand{\laplacian}[0]{\Delta}

\newcommand{\jump}[1]{\llbracket#1\rrbracket}
\newcommand{\tension}[1]{\varsigma_{#1}}

\newcommand{\dL}[1]{\mathrm{d}\mathcal{L}^{#1}}
\newcommand{\dH}[1]{\mathrm{d}\mathcal{H}^{#1}}
\newcommand{\closure}[1]{\overline{#1}}

\newcommand{\bv}[1]{\Vec{#1}}
\newcommand{\bvp}[1]{\mathbf{#1}}

\newcommand{\halfspace}[1]{\mathbb{H}}
\newcommand{\wall}[0]{\mathcal{W}}

\newcommand{\triplejunction}[1]{\mathcal{T}_{#1}}
\newcommand{\triplejunctionIndex}[2]{s^{#1}_{#2}}
\newcommand{\naturalset}[1]{\mathbb{N}_{\leq #1}}
\newcommand{\regionindexm}[1]{r_{#1}^-} 
\newcommand{\regionindexp}[1]{r_{#1}^+} 
\newcommand{\regionindexpm}[1]{r_{#1}^\pm} 
\newcommand{\phaseindexm}[1]{p_{#1}^-} 
\newcommand{\phaseindexp}[1]{p_{#1}^+} 
\newcommand{\region}[1]{\Omega_{#1}} 
\newcommand{\fundSolution}[0]{\Phi}
\newcommand{\smallChemical}[2]{w^{(#1)}_{#2}}
\newcommand{\smallChemicalNoRegion}[1]{w_{#1}}
\newcommand{\curvature}[0]{\varkappa}

\newcommand{\curvatureDisc}[2]{\kappa^h_{#1,#2}}
\newcommand{\charge}[3]{Q^{(#1)}_{#2,#3}}

\newcommand{\singularPoint}[3]{\bv{y}^{(#1)}_{#2,#3}}
\newcommand{\singularPointDummy}[3]{\bv{z}^{(#1)}_{#2,#3}}

\newcommand{\domain}[1]{\Omega_{#1}}
\newcommand{\normal}[2]{\bv{\nu}_{#1,#2}}
\newcommand{\vertex}[2]{\bv{X}_{#1,#2}}
\newcommand{\vertexSeq}[3]{\bv{X}^{(#1)}_{#2,#3}}
\newcommand{\vertexCenter}[2]{\bv{X}^*_{#1,#2}}
\newcommand{\vertexNormal}[2]{\bv{\nu}^h_{#1,#2}}
\newcommand{\vertexNum}[1]{N_{#1}}
\newcommand{\edge}[2]{\sigma_{#1,#2}}
\newcommand{\angleouter}[2]{\phi_{#1,#2}}
\newcommand{\scalarG}[5]{G^{(#1)}_{(#2,#3),(#4,#5)}}
\newcommand{\vectorH}[5]{\bv{H}^{(#1)}_{(#2,#3),(#4,#5)}}
\newcommand{\polygon}[1]{\Gamma^h_{#1}}

\newif\ifrevisionmarkup
\revisionmarkupfalse
\ifrevisionmarkup
  \definecolor{revision}{rgb}{1,0,0}
\else
  \definecolor{revision}{rgb}{0,0,0}
\fi
\newcommand{\fix}[1]{\textcolor{revision}{#1}}

\usepackage{ifthen}
\newcommand{\revisionbibkey}{EGN26-2}  
\newif\ifrevisionbibopen
\let\revisionoldbibitem\bibitem
\renewcommand{\bibitem}[1]{%
  \ifthenelse{\equal{#1}{\revisionbibkey}}%
    {\color{revision}\global\revisionbibopentrue}%
    {\ifrevisionbibopen\normalcolor\global\revisionbibopenfalse\fi}%
  \revisionoldbibitem{#1}%
}

\newcommand{\regionPhaseMap}[0]{\mathscr R_p}
\newcommand{\curveRegionMap}[0]{\mathscr C_r}
\newcommand{\curvePhaseMap}[0]{\mathscr C_p}
\newcommand{\R}[1]{\mathbb R^{#1}}
\newcommand{\polygonLength}[1]{L^h_{#1}}
\newcommand{\unknownsNum}[0]{n_{unk}}
\newcommand{\range}[1]{\operatorname{range} #1}
\newcommand{\ldot}[0]{\dot L}
\newcommand{\lDiscrete}[0]{L^h}
\newcommand{\ldotDiscrete}[0]{\dot L^h}
\newcommand{\regionDiscrete}[1]{\region{#1}^h}
\newcommand{\rotation}[0]{R_{90}}

\title{A Structure-Preserving Method of Fundamental Solutions for the Multi-Phase Mullins--Sekerka Flow}
\author{Tokuhiro Eto\thanks{Universit\'{e} Claude Bernard Lyon 1, CNRS, Centrale Lyon, INSA Lyon, Universit\'{e} Jean Monnet, ICJ UMR5208, 69622 Villeurbanne, France. email:eto@math.univ-lyon1.fr}}
\date{}

\begin{document}

\maketitle

\begin{abstract}
A charge simulation method\fix{, a variant of the method of fundamental solutions,} is applied to approximate the multi-phase Mullins--Sekerka flow in $\R{2}$ and in a
half-plane $\halfspace{2}$ bounded by a Neumann wall $\wall$.
In the underlying mathematical model, interfaces driven by their curvature are coupled through a harmonic chemical-potential field.
\fix{Each} chemical potential is represented by fundamental solutions centered at
charge points off the curve, so no bulk mesh is required. \fix{The scheme treats} curve
networks separating several phases at triple junctions, including phases that occupy more than one
region; on the half-plane boundary, the \fix{homogeneous Neumann} condition is imposed exactly by image charges, and
mobile contacts stay orthogonal to the wall. The discretization is structure-preserving in the sense that every bounded
phase area is conserved to machine precision at the velocity level by a null-space projection of the
discrete area constraints.
The proposed scheme is assessed through a convergence test against an exact three-concentric-circle solution.
\end{abstract}
\medskip
\noindent\textbf{Key words.} multi-phase Mullins--Sekerka flow, method of fundamental solutions, charge simulation method, structure-preserving discretization, triple junction, Neumann boundary condition

\smallskip
\noindent\textbf{MSC codes.} 65M80, 35R37, 53E10, 80A22

\section{Introduction}\label{sec:intro}

When a two-phase mixture is quenched into its miscibility gap, it separates into domains
that subsequently coarsen: large domains grow at the expense of small ones, the total
interfacial length decreases, and the area of each phase is conserved. In the
sharp-interface description of this \textit{late-stage Ostwald ripening} \cite{LS61}, the phase
boundaries move by the Mullins--Sekerka flow \cite{MS63}. Each interface is driven by
its own curvature through a Gibbs--Thomson relation, but the driving force is not
local: at every instant one solves for a harmonic bulk field, the chemical potential,
whose normal-derivative jump across the interface prescribes the normal velocity. The
interfaces communicate only through this bulk field. This non-locality is the
source of the flow's structure: it makes the evolution an $H^{-1/2}$-type
gradient flow of the interfacial energy that dissipates the total length while
preserving the area of each phase \cite{Chen93}. At the same time it is the source
of its numerical difficulty, since a global elliptic problem must be solved before
the geometry can be advanced by even a single step.

Moving from two phases to many turns the interface into a network of curves meeting at
triple junctions, and the evolution of a smooth boundary becomes the evolution of a
network with singular vertices. At each junction, the three incident curves obey a force
balance together with a balance of diffusive fluxes. \fix{We call that force balance the
\textit{Herring--Young condition} \cite{BR93,Herring51}, and use this single name
throughout; when the surface tensions are equal it is the symmetric $120^\circ$
configuration of Young's law.} The present paper is concerned with the numerical approximation of this
multi-phase Mullins--Sekerka flow, both in the whole plane and in a half-plane bounded
by a Neumann wall along which interfaces may slide while meeting the wall orthogonally.

We now briefly review numerical methods for interfacial flows.
\textit{Phase-field methods} regularize the sharp
interface by a thin diffuse layer and evolve the Cahn--Hilliard equation whose
singular limit corresponds to the Mullins--Sekerka flow equation \cite{ABC94,CH58,Pego89,S96}.
\textit{Level-set and threshold-dynamics methods}
\cite{EO15,MBO94,SS11,ZCMO96} represent an evolving interface implicitly as a level set of auxiliary functions in a bulk space.
\textit{Boundary integral methods} \fix{(BIM)} \cite{BB00,BCD95,CKT17,Mayer00,ZCH96} discretize the interface alone,
representing the harmonic field by a layer potential\fix{. L}ater boundary-integral work has addressed multi-component fluids and Ostwald ripening
\cite{AV94,HLS01}. In particular, \textit{recursively compressed inverse preconditioning} (RCIP) \cite{H13,HO08}
\fix{solves integral equations for elliptic problems in a piecewise smooth domain. It
achieves high accuracy even when the boundary includes triple junctions.}
\textit{Parametric finite element methods} \fix{(PFEM)}
\cite{BZ21,BGN07,BGN08,BGN08a,BGN10,BGN15,BGN20,GNZ26,R23} incorporate 
balance conditions around triple junctions naturally into a variational formulation for various surface evolution equations and have good mesh
properties. Moreover, the schemes are unconditionally stable with respect to time steps\fix{. S}ome of these schemes can preserve conservative quantities (e.g., the volume of each phase) at the fully discrete level.
Related structure-preserving parametric finite element schemes have also been developed for
the multi-phase Mullins--Sekerka problem and a degenerate multi-phase Stefan problem with triple junctions \cite{EGN24,EGN26-2,EGN26}.

\fix{Since the bulk field has to be recomputed on a geometry that has just moved, a
representation that carries no bulk mesh is attractive here.}
In this paper, we adopt a \textit{charge simulation method} (CSM), a variant
of the \textit{method of fundamental solutions} (MFS)
\cite{Bogomolny85,FK98,GC99,Katsurada90,KO88}. The idea is to represent each bulk function
as a finite combination of fundamental solutions of the Laplacian centered at
\textit{charge points} placed off the interface. Since every basis function is exactly
harmonic away from its singular point, the Laplace equation is satisfied automatically,
and the only conditions left to enforce are those on the interface itself.
No singular integrals arise, since the charge points are held at a positive distance from the collocation
points, and no bulk mesh is needed. In this sense, we say that the method is bulk mesh-free and is a
boundary-only method, in which the ambient domain has no mesh while the interface is
discretized as a polygonal curve.
For further details about the MFS, we refer the reader to a review article \cite{CH20}.

A numerical scheme based on the CSM for the two-phase Mullins--Sekerka
flow, including contact-angle problems on a wall boundary, has been proposed by the author \cite{E24}.
The contribution of the present paper is to extend that idea to several phases, triple
junctions, phases composed of more than one region, and a Neumann boundary. Since the closest existing multi-phase
scheme \cite{EGN24} discretizes the problem through a fundamentally different, bulk or
whole-domain weak formulation, we do not attempt a direct method-to-method numerical
comparison; instead we validate the present method against the exact three-phase solution
of \cite{EGN24} (Section~\ref{sec:experiments}).

{\color{revision}
Table~\ref{tab:method-comparison} places the three families side by side on the points where
they differ in kind rather than in accuracy on one example. The present method is at its most
useful where a bulk mesh is the expensive part: the ambient domain carries no mesh at any
time, the wall condition costs no unknowns, and no singular integral has to be evaluated.
It is at a disadvantage where implicit time stepping pays. Structure-preserving parametric
finite element schemes are unconditionally stable and several of them conserve the phase
volumes at the fully discrete level, whereas the scheme developed here is advanced
explicitly under $\Delta t = cN^{-2}$ and conserves the phase areas exactly at the level of
the reconstructed velocity, the polygonal areas drifting at higher order
(\Section{sec:analysis}, \Remark{rem:vel-vs-poly}). Its collocation matrix is also severely
ill-conditioned, as is usual for the method of fundamental solutions.
}

\begin{table}[tbp]
  \centering
  {\color{revision}
  \small
  \begin{tabular}{>{\raggedright\arraybackslash}p{0.17\linewidth} >{\raggedright\arraybackslash}p{0.215\linewidth} >{\raggedright\arraybackslash}p{0.195\linewidth} >{\raggedright\arraybackslash}p{0.29\linewidth}}
    \hline
     & BIM / RCIP & PFEM & CSM (present) \\
    \hline
    Triple junctions &
      corner singularities need dedicated treatment; RCIP is accurate there
      \cite{H13,HO08} &
      enter the variational formulation naturally \cite{BGN07,EGN24} &
      imposed geometrically, outside the linear system; curvature corrected locally
      (\Remark{rem:dcorr}) \\
    \hline
    Conditioning &
      well conditioned (second kind) &
      well conditioned, sparse &
      collocation block severely ill-conditioned; area rows via a small Gram matrix
      (\Remark{rem:conditioning}) \\
    \hline
    Neumann wall &
      through the Green's function &
      weak; bulk mesh must fit the wall &
      exact by images on a straight wall, no extra unknowns \\
    \hline
    Cost per degree of freedom &
      interface only; dense, singular quadrature &
      bulk mesh; step not resolution-bound &
      interface only; no bulk mesh or singular quadrature; dense $O(N^2)$ per explicit
      step \\
    \hline
  \end{tabular}}
  \caption{\fix{The families of numerical methods.
    The entries are structural properties of each method.}}
  \label{tab:method-comparison}
\end{table}

Our discretization is guided by a second principle: structure preservation.
Since the continuous Mullins--Sekerka flow preserves the area of each phase, even a slow numerical drift could be mistaken for physical mass transfer.
We therefore treat the
discrete area-conservation identities as hard constraints and project the reconstructed
velocity
onto the corresponding subspace, so that the discrete area fluxes vanish to machine
precision (velocity-level conservation, Section~\ref{sec:analysis}), independently of the
residual of the field solve; the residual drift of the polygonal areas after time
stepping is of higher order and is reported in Section~\ref{sec:experiments}. Two further
ingredients are handled in the same way. The discrete curvature that feeds the Gibbs--Thomson
condition is unreliable on the edges adjacent to a triple junction, where a naive
evaluation using the junction vertex fails to converge to the interface curvature and
biases the reconstructed velocity near the junctions; we correct it locally. And the
homogeneous Neumann condition on a straight wall
is imposed exactly by reflecting each charge in the wall, at the cost of no additional
unknowns.

{\color{revision}
The objective of this paper is to construct a bulk mesh-free, structure-preserving
numerical scheme for the multi-phase Mullins--Sekerka flow in the plane and in the
half-plane. More specifically, we aim (i) to build a CSM approximation for a network of
curves meeting at triple junctions in the whole plane; (ii) to extend it to a half-plane
bounded by a Neumann wall, along which contact points may move while the interfaces meet
the wall orthogonally; and (iii) to ensure that the discrete scheme respects the analytical
invariants of the continuum Mullins--Sekerka flow, in particular that the area of every
bounded phase is conserved exactly at the velocity level rather than up to a discretization
error.
}

The main contributions of this paper are the following.
\begin{itemize}
  \item {\color{revision}\textbf{Objective:} to approximate the multi-phase Mullins--Sekerka flow in
        the whole plane, including the conditions that hold at a triple junction.}\newline\newline
        {\color{revision}\textbf{Result:} a CSM scheme} with two-sided least-squares collocation of the
        Gibbs--Thomson and continuity
        conditions and a geometric treatment of the Herring--Young condition at triple
        junctions\fix{,} extended to a Neumann wall and
        orthogonal boundary contact.
  \item {\color{revision}\textbf{Objective:} to make the discrete flow inherit the exact area
        conservation of the continuous one.}\newline\newline
        {\color{revision}\textbf{Result:}} every bounded phase area \fix{is conserved} at the velocity
        level \fix{to machine precision},
        enforced by projection onto the null space of the discrete area constraints, with
        the accompanying discrete structure-preservation statements\fix{;}
        \fix{t}he approximate flow \fix{is validated against} an exact
        three-concentric-circle solution.
  \item {\color{revision}\textbf{Objective:} to let a shrinking region leave the computation without
        disturbing the junction network.}\newline\newline
        {\color{revision}\textbf{Result:}} a region enclosed by a single closed interface and incident to no
        junction, once it has shrunk below the mesh resolution, \fix{is} removed \fix{and}
        the phase-area constraints are re-baselined; the induced area change equals the removed
        residual\fix{, t}he junction connectivity is not altered, and no curve splicing or vertex
        surgery is involved.
\end{itemize}

\section{Problem formulation}\label{sec:formulation}
In this paper, we consider the multi-phase Mullins--Sekerka flow equation:
\begin{subequations}
    \label{system:all}
\begin{align}
    &\laplacian \bvp{w}(\cdot, t) = \bvp{0}\qquad&&\mbox{in}\quad\mathbb{R}^d\setminus\Gamma(t),\quad t > 0,\label{system:laplace}\\
    &\bvp{w}(\cdot, t)\cdot\jump{\bvp{\chi}} = \tension{}\curvature\qquad &&\mbox{on}\quad\Gamma(t),\quad t > 0,\label{system:gibbs-thomson}\\
    &\jump{\nabla\bvp{w}(\cdot,t)}\cdot\bv{\nu_\Gamma} = -V\jump{\bvp{\chi}} \qquad&&\mbox{on}\quad\Gamma(t),\quad t > 0,\label{system:motion}\\
    &\nabla\bvp{w}(\bv{x}, t) = O\left(\frac{1}{|\bv{x}|^2}\right) \qquad&&\mbox{as}\quad |\bv{x}|\to\infty,\quad t > 0,\label{system:neumann}\\
    &\sum_{\ell=1}^3\tension{\triplejunctionIndex{k}{\ell}}\bv{\mu}_{\triplejunctionIndex{k}{\ell}} = \bv{0}\qquad&&\mbox{on}\quad\triplejunction{k}(t),\quad 1\leq k\leq I_T,\quad t > 0,\label{system:young}\\
    &\jump{\bvp{w}(\cdot,t)} = \bvp{0}\qquad&&\mbox{on}\quad\Gamma(t),\quad t > 0,\label{system:continuity}\\
    &\Gamma(0) = \Gamma_0\label{system:initial},
\end{align}
\end{subequations}
where for each $t > 0$, $\bvp{w}(\cdot, t):\mathbb{R}^d\setminus\Gamma(t)\to\mathbb{R}^{I_P}\,(I_P\geq 2)$ is a vector-valued function of harmonic functions which indicates the chemical potential of each phase;
a curve network $\Gamma$ is composed of curves $\Gamma_i\,(1\leq i\leq I_S)$ with $I_S \geq 1$, and separates the ambient space $\mathbb{R}^d$ into one unbounded domain $\region{I_R}$ and several bounded regions $\region{1},\cdots,\region{I_R-1}$ with $I_R\geq 2$.
We assume that the chemical potentials $\bvp{w}(\cdot,t) = (w_1(\cdot,t),\cdots,w_{I_P}(\cdot,t))^\top$ satisfy the zero-sum condition in $\mathbb{R}^d\setminus\Gamma(t)$, i.e., $\sum_{p=1}^{I_P}w_p(\cdot,t) = 0$.
The function $\bvp{\chi}(\cdot,t):\mathbb{R}^d\setminus\Gamma(t)\to\{0,1\}^{I_P}$ denotes the
characteristic function of the phases; the symbols $\tension{}$ and $\curvature$ respectively denote the surface tension coefficient and the curvature of the curve $\Gamma$ in the direction of the unit normal vector $\bv{\nu}_\Gamma$,
and $\tension{}$ is assumed to be piecewise constant.
With this convention, $\curvature_i := \curvature\lfloor_{\Gamma_i}$ is negative when $\Gamma_i$ is convex (for instance, $\curvature_i < 0$ for a counter-clockwise convex closed curve).
Then, the condition \eqref{system:gibbs-thomson} encodes the Dirichlet condition, the so-called \textit{Gibbs--Thomson law}, to solve the Laplace equations \eqref{system:laplace};
the symbol $\jump{\cdot}$ denotes the jump of a quantity (allowed to be vector-valued) across a curve $\Gamma_i$ defined by
\begin{equation*}
    \jump{q}(\bv{x}) := \lim_{\varepsilon\to 0} \left\{q(\bv{x}+\varepsilon\bv{\nu}_{\Gamma_i}) - q(\bv{x}-\varepsilon\bv{\nu}_{\Gamma_i})\right\}\qquad\mbox{for}\quad\bv{x}\in\Gamma_i.
\end{equation*}
The symbol $V_i := V\lfloor_{\Gamma_i}$ denotes the normal velocity of $\Gamma_i$ in the direction $\bv{\nu}_{\Gamma_i}$ which is determined by the jump of the normal derivative of the two chemical potential functions \eqref{system:motion},
which correspond to the phases where the curve $\Gamma_i$ separates;
for each $1\leq k\leq I_T$, $(\triplejunctionIndex{k}{1},\triplejunctionIndex{k}{2},\triplejunctionIndex{k}{3})$ with $1\leq \triplejunctionIndex{k}{1}<\triplejunctionIndex{k}{2}<\triplejunctionIndex{k}{3}\leq I_S$
is the triplet of the curve indices at which $\Gamma_{\triplejunctionIndex{k}{1}}$, $\Gamma_{\triplejunctionIndex{k}{2}}$, and $\Gamma_{\triplejunctionIndex{k}{3}}$ compose the triple junction point $\triplejunction{k}(t)$.
A curve network without triple junctions is also possible and can be understood in the case $I_T = 0$.
The vector $\bv{\mu}_i\,(1\leq i\leq I_S)$ designates the co-normal vector field on the curve $\Gamma_i$, so it is tangential to $\Gamma_i$, at its endpoints;
the \fix{\textit{Herring--Young condition}} should be satisfied which is encoded by \eqref{system:young}.
The gradient of all chemical potentials is required to decay at the rate $1/|\bv{x}|^2$ as $|\bv{x}|\to\infty$ \eqref{system:neumann},
and this is an alternative to the pure Neumann boundary condition.
Each chemical potential should be continuous across the boundary \eqref{system:continuity}.
To close the system, the initial curve network $\Gamma_0$ is given in \eqref{system:initial}.

Throughout this paper, we restrict ourselves to the planar case $d=2$; accordingly $\mathbb{R}^d = \mathbb{R}^2$, $\dL{d} = \dL{2}$, and $\dH{d-1} = \dH{1}$ wherever the dimension-general notation appears below.
Moreover, for notation of vectors and vector-valued functions, we shall use the bold fonts and vector symbol to distinguish $\R{I_P}$-valued from $\R{2}$-valued.

\begin{rem}\label{rem:normal-invariant}
    It is easily seen that the system \eqref{system:all} is invariant with respect to a choice of the direction of $\bv{\nu}_{\Gamma}$.
    Indeed, if the sign of $\bv{\nu}_\Gamma$ changes, then so do $\jump{\chi}$, $\curvature$, $\jump{\nabla\bvp{w}}$, and $V$.
\end{rem}

\begin{rem}\label{rem:sumzero}
    The zero-sum condition on $\bvp{w}$ is necessary to guarantee the uniqueness of the solution to \eqref{system:all}.
    Indeed, if $\bvp{w}$ is a solution to the system, then $\bvp{w} + c(1,\dots,1)^\top,\,c\in\R{}$ also satisfies all equations since
    the Gibbs--Thomson law \eqref{system:gibbs-thomson} is described in terms of the difference of two chemical potentials.
    In the pure Neumann boundary problems, the zero-sum condition has been used to guarantee the uniqueness of a solution to a linear system in the literature of the parametric finite element method (see \cite[Theorem 4.1]{EGN24}).
\end{rem}
\begin{rem}\label{rem:cont}
    If $\bvp{w}(\cdot,t)\in W^{1,p}_{loc}(\mathbb{R}^2)^{I_P}$ for some $p > 2$,
    then the continuity condition \eqref{system:continuity} is automatically satisfied by the embedding $W^{1,p}(\Omega)\hookrightarrow C^{\alpha}(\closure{\Omega})$
    with $\alpha := 1 - 2/p$ for any bounded smooth domain $\Omega$ in $\mathbb{R}^2$ thanks to the Morrey theorem.
\end{rem}
\begin{rem}\label{rem:young}
The condition \eqref{system:young} stems from the force balance at the triple junctions.
\fix{In this paper, we call this the
\textit{Herring--Young condition} throughout.}
To ensure this, we require that for every $k\in\naturalset{I_T}$,
\begin{equation*}
    \tension{\triplejunctionIndex{k}{1}} \leq \tension{\triplejunctionIndex{k}{2}} + \tension{\triplejunctionIndex{k}{3}},\quad
    \tension{\triplejunctionIndex{k}{2}} \leq \tension{\triplejunctionIndex{k}{3}} + \tension{\triplejunctionIndex{k}{1}},\quad\mbox{and}\quad
    \tension{\triplejunctionIndex{k}{3}} \leq \tension{\triplejunctionIndex{k}{1}} + \tension{\triplejunctionIndex{k}{2}}.
\end{equation*}
These three inequalities are exactly the solvability condition for the force balance \eqref{system:young}.
They are necessary and sufficient for the three surface-tension vectors to close into a triangle, hence for an equilibrium set of junction angles to exist.
If $\theta_{ab}$ denotes the angle formed at a junction between the interfaces with
surface tensions $\tension{a}$ and $\tension{b}$, opposite the interface with tension
$\tension{c}$, then
\begin{equation*}
    \cos\theta_{ab} = \frac{\tension{c}^2-\tension{a}^2-\tension{b}^2}{2\,\tension{a}\tension{b}}.
\end{equation*}
In particular, all three junction angles are $120^\circ$ when the surface tensions are equal.
\end{rem}
\begin{rem}\label{rem:wellposed}
For the multi-phase Mullins--Sekerka flow, a global weak solution has been established by
Bronsard, Garcke and Stoth \cite{BGS98} through an implicit time discretization;
well-posedness of strong solutions in the presence of triple junctions and boundary contacts
remains open. This paper concerns a numerical approximation of the flow, not its
well-posedness. In the two-phase case, a global weak solution (BV solution)
has been established by Luckhaus and Sturzenhecker \cite{LS95} under a no mass-loss assumption on discrete solutions constructed by a minimizing movement scheme.
This assumption has been removed by R\"oger \cite{R05} by means of the notion of \textit{varifolds}.
For well-posedness of the two-phase Mullins--Sekerka problem in an unbounded domain in $\R{2}$,
we refer the reader to Escher, Matioc, and Matioc \cite{EMM24}.
\end{rem}

\section{Assumptions on curve network}\label{sec:network}
In this section, we introduce the mathematical notation and assumptions for the curve network.
For $K\in\mathbb N$, let $\naturalset{K}:= \{1,\cdots,K\}$, and we use the convention that $\naturalset{0} := \emptyset$.

\begin{asm}[Curve network]
    The curve network $\Gamma$ is a piecewise parametrized curve which separates the whole space $\mathbb{R}^2$ into regions $\region{1},\cdots,\region{I_R}\,(I_R\geq 2)$.
The curve network $\Gamma$ is composed of the parametrized curves $\Gamma_1,\cdots,\Gamma_{I_S}\,(I_S\geq 1)$,
and each curve is automatically oriented by the map $\Gamma_i:(0,1)\to\mathbb{R}^2$.
Throughout this paper, the direction of the unit normal vector $\normal{i}{j}$
is supposed to be so that it points to the right-hand side of the curve.
Namely, we assume that
\begin{equation*}
    \Gamma = \bigcup_{i=1}^{I_S}\Gamma_i\qquad\mbox{and}\qquad\mathbb{R}^d = \bigcup_{r=1}^{I_R}\Omega_r\cup\Gamma,\quad\Omega_{r_1}\cap\Omega_{r_2} = \emptyset\quad\mbox{for}\quad r_1\neq r_2.
\end{equation*}
In particular, we assume that the outermost region is $\Omega_{I_R}$, and this region is supposed to be unbounded.
The curve network $\Gamma$ is not necessarily connected. In other words, it possibly contains several disconnected components.
Each curve in the curve network $\Gamma$ is supposed to be either an open curve or a closed curve.
Here, $\Gamma_i$ is said to be open (resp. closed) if $\Gamma_i(0) \neq \Gamma_i(1)$ (resp. $\Gamma_i(0) = \Gamma_i(1)$).
Moreover, we assume that all curves in $\Gamma$ do not have self intersections, i.e., $\Gamma_i(t_1) = \Gamma_i(t_2)$ implies $t_1 = t_2$ unless $\Gamma_i$ is closed with $t_1 = 0$ and $t_2 = 1$.
We also assume that no curve intersects another curve in $\Gamma$ except at triple junctions.
We note that each closed curve in $\Gamma$ can only enclose a single region. On the other hand,
with open curves, we need several curves to enclose a region.

The curve network $\Gamma$ is supposed to have triple junctions $\triplejunction{1},\cdots,\triplejunction{I_T}\,(I_T\geq 0)$.
Each triple junction $\triplejunction{k}$ is the meeting point of three curves $\Gamma_{\triplejunctionIndex{k}{1}}$, $\Gamma_{\triplejunctionIndex{k}{2}}$, 
and $\Gamma_{\triplejunctionIndex{k}{3}}$ with $1\leq \triplejunctionIndex{k}{1}<\triplejunctionIndex{k}{2}<\triplejunctionIndex{k}{3}\leq I_S$.
We note that a curve network having no triple junctions is also considered in our paper;
in this case, we suppose that the Herring--Young condition \eqref{system:young} is trivially satisfied.
\end{asm}

To describe a curve network precisely in a mathematical manner,
we now introduce mappings which relate indices of curves, regions and phases.

\begin{asm}[Region to phase map $\regionPhaseMap$]
Any region is assigned to one phase.
To describe this, we introduce a map $\regionPhaseMap: \naturalset{I_R}\to\naturalset{I_P}$ such that the region $\domain{r}$ is occupied by the phase $\regionPhaseMap(r)$.
Conversely, the $p$-th phase is composed of the regions $\domain{r}$ with $r\in \regionPhaseMap^{-1}(p)$.
We assume that the map $\regionPhaseMap$ is surjective. In other words, all phases are composed of at least one region.
According to the assumption on $\region{I_R}$, we require that this outer region belongs to the $I_P$-th phase, namely $\regionPhaseMap(I_R) = I_P$.
In other words, the $I_P$-th phase is the unique phase containing the unbounded region; it may in addition own bounded regions.
\end{asm}

\begin{asm}[Curve to region map $\curveRegionMap$]
Any curve is supposed to separate exactly two regions.
To prescribe this, we introduce the \textit{Curve to Region map} which defines the correspondence between the curve and the region.
Precisely speaking, it is a map $\curveRegionMap: \naturalset{I_S}\to\naturalset{I_R}^2$
such that $\curveRegionMap(s) = \left(\regionindexm{s},\regionindexp{s}\right)$ means that the curve $\Gamma_s$ lies between the regions $\Omega_{\regionindexm{s}}$ and $\Omega_{\regionindexp{s}}$.
Moreover, this relation stresses that the normal vector $\bv{\nu}_{\Gamma_s}$ points from $\Omega_{\regionindexm{s}}$ to $\Omega_{\regionindexp{s}}$.
Conversely, using $\curveRegionMap$, for each $r\in\naturalset{I_R}$, the curves which enclose the $r$-th domain can be identified by the following index set:
\[
\left\{s\in\naturalset{I_S}\biggm| r = \regionindexp{s}\quad\text{or}\quad r = \regionindexm{s}\right\}.
\]
\end{asm}

\begin{asm}[Curve to phase map $\curvePhaseMap$]
We compose the maps $\curveRegionMap$ and $\regionPhaseMap$ to define the \textit{Curve to Phase map} $\curvePhaseMap: \naturalset{I_S}\to\naturalset{I_P}^2$ and write as
$\curvePhaseMap(s) = (\phaseindexm{s},\phaseindexp{s})$ for $s\in\naturalset{I_S}$. Using this notation, we have $(\phaseindexm{s}, \phaseindexp{s}) = (\regionPhaseMap(\regionindexm{s}),\regionPhaseMap(\regionindexp{s}))$.
We stress that $\phaseindexm{s}\neq \phaseindexp{s}$ for all $s\in\naturalset{I_S}$ so that any two adjacent regions must belong to different phases.

Conversely, using $\curvePhaseMap$, for each $p\in\naturalset{I_P}$, the curves which enclose the $p$-th phase can be identified by the following index set:
\[
\left\{s\in\naturalset{I_S}\biggm| p = \phaseindexp{s}\quad\text{or}\quad p = \phaseindexm{s}\right\}.
\]
\end{asm}

We show a representative curve network in Figure~\ref{fig:network-schematic} composed of
three open curves $\Gamma_i\,(i=1,\,2,\,3)$ meeting at the triple junctions $\triplejunction{1}$ and $\triplejunction{2}$
and one closed curve $\Gamma_4$. The curves separate the plane into the bounded regions $\region{r}\,(r=1,\,2,\,3)$, and the
unbounded exterior region $\region{4}$. We note that, in the displayed case,
the region-to-phase map $\regionPhaseMap$ is non-injective:
phase $2$ occupies the two disconnected regions $\region{2}$ and $\region{3}$ (same color),
and $\regionPhaseMap^{-1}(2)=\{2,3\}$. For the curve $\Gamma_3$, the curve-to-region map gives
$\curveRegionMap(3)=(\regionindexm{3},\regionindexp{3})=(1,2)$, and the unit normal
$\bv{\nu}_{\Gamma_3}$ points from $\region{1}$ to $\region{2}$.

\begin{figure}[H]
  \centering
  \begin{tikzpicture}[>={Stealth[length=2mm]}, line join=round, scale=1.1]
    \definecolor{phaseAcol}{RGB}{230,159,0}   
    \definecolor{phaseBcol}{RGB}{86,180,233}  
    \definecolor{phaseCcol}{RGB}{0,158,115}   
    \coordinate (T1) at (0,1.05);
    \coordinate (T2) at (0,-1.05);
    \fill[phaseCcol!14] (-3.05,-1.95) rectangle (5.65,1.95);
    \fill[phaseAcol!35] (T1) to[out=150,in=210,looseness=3.5] (T2) to[out=90,in=-90,looseness=0.8] (T1);  
    \fill[phaseBcol!35] (T1) to[out=30,in=330,looseness=3.5]  (T2) to[out=90,in=-90,looseness=0.8] (T1);  
    \fill[phaseBcol!35] (4.0,0) circle (1.0);                                                             
    \draw[thick] (T1) to[out=150,in=210,looseness=3.5] (T2);  
    \draw[thick] (T1) to[out=30,in=330,looseness=3.5]  (T2);  
    \draw[thick] (T1) to[out=-90,in=90,looseness=0.8] (T2);   
    \draw[thick] (4.0,0) circle (1.0);                        
    \foreach \p in {T1,T2}{\fill[white] (\p) circle (2.7pt); \fill (\p) circle (1.7pt);}
    \draw[->,black,thick] (0,0.30) -- (0.60,0.30) node[above,black,inner sep=1pt] {$\bv{\nu}_{\Gamma_3}$};
    \path (T1) to[out=150,in=210,looseness=3.5] node[pos=0.5,inner sep=0,outer sep=0](gL){} (T2);
    \path (T1) to[out=30,in=330,looseness=3.5]  node[pos=0.5,inner sep=0,outer sep=0](gR){} (T2);
    \draw[->,black,thick] (gL) -- ($(gL)+(-0.5,0)$) node[left,inner sep=1.5pt]  {$\bv{\nu}_{\Gamma_1}$};
    \draw[->,black,thick] (gR) -- ($(gR)+(0.5,0)$)  node[right,inner sep=1.5pt] {$\bv{\nu}_{\Gamma_2}$};
    \draw[->,black,thick] ($(4,0)+(-45:1)$) -- ($(4,0)+(-45:1.46)$) node[below right,inner sep=1pt] {$\bv{\nu}_{\Gamma_4}$};
    \node at (-1.20,-0.05) {$\region{1}$};
    \node at (1.20,-0.05)  {$\region{2}$};
    \node at (4.0,0)   {$\region{3}$};
    \node at (2.55,1.35) {$\region{4}$};
    \node[above=1pt] at (T1) {$\triplejunction{1}$};
    \node[below=1pt] at (T2) {$\triplejunction{2}$};
    \node[left] at (-1.72,0.62) {$\Gamma_1$};
    \node[right] at (1.72,0.62)  {$\Gamma_2$};
    \node[right,inner sep=1.5pt] at (0.05,-0.62) {$\Gamma_3$};
    \node[above] at (4.0,1.0) {$\Gamma_4$};
    \begin{scope}[shift={(-3.05,-2.70)}]
      \fill[phaseAcol!35,draw=black!45] (0,0) rectangle (0.34,0.24);
      \node[right,inner sep=2.5pt] at (0.34,0.12) {\footnotesize\color{revision} phase $1$};
      \fill[phaseBcol!35,draw=black!45] (2.10,0) rectangle (2.44,0.24);
      \node[right,inner sep=2.5pt] at (2.44,0.12) {\footnotesize\color{revision} phase $2$};
      \fill[phaseCcol!14,draw=black!45] (4.20,0) rectangle (4.54,0.24);
      \node[right,inner sep=2.5pt] at (4.54,0.12) {\footnotesize\color{revision} phase $3$};
    \end{scope}
  \end{tikzpicture}
  \par\smallskip
  {\color{revision}\footnotesize
  \begin{tabular}{r|cccc}
    $i\in\naturalset{I_S}$ & $1$ & $2$ & $3$ & $4$ \\
    \hline
    $\curveRegionMap(i)=(\regionindexm{i},\regionindexp{i})$ & $(1,4)$ & $(2,4)$ & $(1,2)$ & $(3,4)$ \\
    $\curvePhaseMap(i)=(\phaseindexm{i},\phaseindexp{i})$ & $(1,3)$ & $(2,3)$ & $(1,2)$ & $(2,3)$ \\
  \end{tabular}\qquad
  \begin{tabular}{r|ccc}
    $p\in\naturalset{I_P}$ & $1$ & $2$ & $3$ \\
    \hline
    $\regionPhaseMap^{-1}(p)$ & $\{1\}$ & $\{2,3\}$ & $\{4\}$ \\
  \end{tabular}}
  \caption{A curve network $\Gamma$ in the case $I_P = 3$, $I_R = 4$, $I_S = 4$, and $I_T=2$.
    \fix{The tables list every value of the maps
    $\curveRegionMap$, $\curvePhaseMap$, and $\regionPhaseMap^{-1}$ for this network.}}
  \label{fig:network-schematic}
\end{figure}

\section{Properties of classical solutions}\label{sec:classical}
In this section, we show two important properties of classical solutions to the system \eqref{system:all}.
\fix{Throughout this section, for a region $\mathcal{R}\subset\R{d}$ with a piecewise smooth boundary, we write $\bv{\nu}_{\mathcal{R}}$ for the outward unit normal vector field of $\partial\mathcal{R}$.}
We begin with the curve shortening property.
We can find a similar statement in \cite[Proposition 2.1]{EGN24}, although
we need some work to extend the proof to the case of the whole space $\mathbb{R}^d$.
\begin{prop}\label{prop:csp}
Assume that $(\bvp{w}(\cdot, t),\Gamma(t))$ is a smooth solution to the system \eqref{system:all}.
Assume further that the Dirichlet energy is finite, i.e., $\nabla\bvp{w}(\cdot,t)\in L^{2}(\mathbb{R}^d)^{I_P\times d}$ for a.e. $t > 0$.
Then, the length of the curve $\Gamma(t)$ is non-increasing in time. Precisely speaking, we have
\begin{equation*}
    \frac{d}{dt}\sum_{i=1}^{I_S}\tension{i}|\Gamma_i(t)| \leq -\int_{\mathbb{R}^2}|\nabla\bvp{w}(\cdot,t)|^2\,\dL{d}\qquad\mbox{for}\quad t > 0.
\end{equation*}
\end{prop}
\begin{proof}
We take $R>0$ so large that $\Gamma_i(t)\subset B_R(0)$ \fix{for every $i\in\naturalset{I_S}$}. Since the curvature is the first variation of the length of the curve, we have
\begin{multline}\label{eq:csp1}
    \frac{d}{dt}\sum_{i=1}^{I_S}\tension{i}|\Gamma_i(t)| = \sum_{i=1}^{I_S}\int_{\Gamma_i(t)}\fix{-}V_i\tension{i}\curvature_i\,\dH{d-1} = \sum_{i=1}^{I_S}\int_{\Gamma_i(t)}\fix{-}V_i(\bvp{w}\cdot\jump{\bvp{\chi}})\,\dH{d-1}\\
    =\sum_{i=1}^{I_S}\int_{\Gamma_i(t)}\bvp{w}\cdot(\fix{-}V_i\jump{\bvp{\chi}})\,\dH{d-1} = \sum_{i=1}^{I_S}\int_{\Gamma_i(t)} \bvp{w}\cdot \jump{\nabla\bvp{w}}\bv{\nu}_i\,\dH{d-1}.
\end{multline}
For each $p\in\naturalset{I_P}$, let $\smallChemical{r}{p} := \smallChemicalNoRegion{p}\lfloor_{\region{r}}$ for $r\in\regionPhaseMap^{-1}(p)$.
Then, we have
\begin{multline}\label{eq:csp2}
    -\int_{B_R(0)}|\nabla \smallChemicalNoRegion{p}|^2\,\dL{d} = -\sum_{r=1}^{I_R-1} \int_{\domain{r}} \left|\nabla \smallChemical{r}{p}\right|^2\,\dL{d} - \int_{B_R(0)\cap\domain{I_R}} \left|\nabla \smallChemical{I_R}{p}\right|^2\,\dL{d}\\
    = - \sum_{r = 1}^{I_R} \int_{\partial\Omega_{r}} \smallChemical{r}{p}\nabla\smallChemical{r}{p}\cdot\bv{\nu}_{\Omega_r}\,\dH{d-1}  - \int_{\partial B_R(0)} \smallChemical{I_R}{p}\nabla\smallChemical{I_R}{p}\cdot\fix{\bv{\nu}_{B_R(0)}}\,\dH{d-1}\\
    = \sum_{i=1}^{I_S} \int_{\Gamma_i(t)} \smallChemicalNoRegion{p}\cdot\jump{\nabla\smallChemicalNoRegion{p}}\bv{\nu}_i\,\dH{d-1}  - \int_{\partial B_R(0)} \smallChemical{I_R}{p}\nabla\smallChemical{I_R}{p}\cdot\fix{\bv{\nu}_{B_R(0)}}\,\dH{d-1}.
\end{multline}
Summing up \eqref{eq:csp2} over $p\in\naturalset{I_P}$, we obtain
\begin{equation}\label{eq:csp3}
    \begin{aligned}
    -\int_{B_R(0)}|\nabla\bvp{w}(\cdot,t)|^2\,\dL{\fix{d}} &= \sum_{i=1}^{I_S}\int_{\Gamma_i(t)}\bvp{\smallChemicalNoRegion{}}\cdot\jump{\nabla\bvp{\smallChemicalNoRegion{}}}\bv{\nu}_i\,\dH{d-1}\\
    &\quad - \sum_{p=1}^{I_P}\int_{\partial B_R(0)}\smallChemical{I_R}{p}\nabla\smallChemical{I_R}{p}\cdot\fix{\bv{\nu}_{B_R(0)}}\,\dH{d-1}.
    \end{aligned}
\end{equation}
Here, we have invoked the fact that $\laplacian\smallChemical{r}{p} = 0$ in $\domain{r}$ to obtain the last equality.
We recall from \cite[Lemma 3]{E24} that $\smallChemicalNoRegion{p}$ is bounded thanks to the assumption that its gradient is $O(1/|\bv{x}|^2)$ as $|\bv{x}|\to\infty$.
Hence, we can estimate the second term on the right-hand side of the above equality as follows:
\begin{multline}\label{eq:csp4}
    \left|\int_{\partial B_R(0)} \smallChemical{I_R}{p}\nabla\smallChemical{I_R}{p}\cdot\fix{\bv{\nu}_{B_R(0)}}\,\dH{d-1}\right| \leq C_1 \int_{\partial B_R(0)} \left|\nabla\smallChemical{I_R}{p}\right|\,\dH{d-1}\\
    \leq C_1\cdot\frac{C_2}{R^2}\int_{\partial B_R(0)}\,\dH{1} \leq \frac{2\pi C_1C_2}{R},
\end{multline}
where $C_1$ and $C_2$ are positive constants such that $|\smallChemicalNoRegion{p}|\leq C_1$ and $|\nabla\smallChemicalNoRegion{p}(\bv{x})|\leq C_2 / |\bv{x}|^2$ for all $\bv{x}\in\mathbb{R}^2$ with $|\bv{x}|\geq R$.
We now combine \eqref{eq:csp1}, \eqref{eq:csp3}, and \eqref{eq:csp4} to obtain
\begin{equation*}
    \frac{d}{dt}\sum_{i=1}^{I_S}\tension{i}|\Gamma_i(t)| \leq -\int_{B_R(0)}|\nabla\bvp{w}(\cdot,t)|^2\,\dL{\fix{d}} + \frac{2\pi C_1C_2 I_P}{R}.
\end{equation*}
Sending $R\to \infty$, we obtain the desired inequality.
\end{proof}

Next, we provide a proof of the area-preserving property.
Again, this property is shown in \cite[Proposition 2.2]{EGN24}, although the proof is carried out in a specific three-phase case.
\begin{prop}\label{prop:ap}
    Assume that $(\bvp{w}(\cdot,t),\Gamma(t))$ is a smooth solution to the system \eqref{system:all}.
    Then, the area of each phase, except for the $I_P$-th phase, is preserved in time. Namely, it holds that
    \begin{equation*}
        \frac{d}{dt}\sum_{r\in \regionPhaseMap^{-1}(p)}|\Omega_r(t)| = 0\qquad\text{for all}\quad t > 0\quad\text{and}\quad p\in\naturalset{I_P-1}.
    \end{equation*}
\end{prop}
\begin{proof}
    Take $R > 0$ so large that $\Gamma\subset B_R(0)$. We deduce from the Laplace equation \eqref{system:laplace} that
    \begin{align}
        \label{eq:ap-1}
        0 &= \int_{\region{I_R}\cap B_R(0)}\laplacian\smallChemical{I_R}{p}\dL{2} + \sum_{r = 1}^{I_R-1}\int_{\region{r}}\laplacian\smallChemical{r}{p}\dL{2}\notag\\
        &= \int_{\partial B_R(0)} \nabla \smallChemical{I_R}{p}\cdot\fix{\bv{\nu}_{B_R(0)}}\dH{1} + \sum_{r = 1}^{I_R}\int_{\partial \region{r}}\nabla\smallChemical{r}{p}\cdot\fix{\bv{\nu}_{\region{r}}}\dH{1},
    \end{align}
    We encode the second term in terms of the index $i\in\naturalset{I_S}$.
    \begin{align}
        \label{eq:ap-2}
        \sum_{r = 1}^{I_R}\int_{\partial\region{r}}\nabla\smallChemical{r}{p}\cdot\fix{\bv{\nu}_{\region{r}}}\dH{1}
        &= \sum_{r=1}^{I_R}\left(\sum_{\substack{i\in\naturalset{I_S}\\r = \regionindexm{i}}}\int_{\Gamma_i}\nabla\smallChemical{r}{p}\cdot\bv{\nu}_i\dH{1} + \sum_{\substack{j\in\naturalset{I_S}\\r = \regionindexp{j}}}\int_{\Gamma_j}\nabla\smallChemical{r}{p}\cdot(-\bv{\nu}_j)\dH{1}\right)\notag\\
        &=\sum_{i=1}^{I_S}\left(\sum_{\substack{r\in\naturalset{I_R}\\r = \regionindexm{i}}}\int_{\Gamma_i}\nabla\smallChemical{r}{p}\cdot\bv{\nu}_i\dH{1} + \sum_{\substack{s\in\naturalset{I_R}\\ s=\regionindexp{i}}}\int_{\Gamma_i}\nabla\smallChemical{s}{p}\cdot(-\bv{\nu}_i)\dH{1}\right)\notag\\
        &=\sum_{i=1}^{I_S}\int_{\Gamma_i}-\jump{\nabla w_p}\bv{\nu}_i\dH{1} = \sum_{i=1}^{I_S}\int_{\Gamma_i}V_i\jump{\chi_p}\dH{1}.
    \end{align}
    Here, we have invoked the motion law \eqref{system:motion} to obtain the last equality.
    We observe that for each $r\in\regionPhaseMap^{-1}(p)$, $V_i\jump{\chi_p}$ corresponds to the inward normal velocity of $\Gamma_i$.
    Therefore, combining \eqref{eq:ap-1} and \eqref{eq:ap-2} together with the decay condition \eqref{system:neumann},
    we deduce that
    \begin{equation*}
        -\frac{d}{dt}\sum_{r\in\regionPhaseMap^{-1}(p)}|\region{r}(t)| + O\left(\frac{1}{R}\right) = 0\qquad\text{as}\quad R\to\infty.
    \end{equation*}
    This concludes the proof.
\end{proof}

\section{Spatial discretization}
\label{sec:spatial_discretization}

To approximate the parametrized curves $\Gamma_i:[0,1]\to\mathbb R^2\,(i\in\naturalset{I_S})$ in a curve network $\Gamma$, we follow the strategy employed in \cite{E24}.
Namely, each curve $\Gamma_i$ is approximated by a polygonal curve $\polygon{i}$ by using ordered vertices $\vertex{i}{1},\cdots,\vertex{i}{\vertexNum{i}}$ defined by

\begin{equation*}
    \vertex{i}{j} :=
    \begin{cases}
       \Gamma_i\left(\frac{j-1}{N_i}\right) & \qquad \text{if}\quad \Gamma_i\ \text{is closed},\\
       \Gamma_i\left(\frac{j-1}{N_i-1}\right) & \qquad \text{if}\quad \Gamma_i\ \text{is open}.
    \end{cases}
\end{equation*}
\noindent
We let
\begin{equation*}
    \edge{i}{j} : [0,1] \ni t\mapsto (1-t)\vertex{i}{j-1} + t\vertex{i}{j} \in \R{2}\qquad\text{for}\quad 1\leq j\leq N_i.
\end{equation*}
From now on, we identify the image of $\edge{i}{j}$ with $\edge{i}{j}$ itself.
Then, we define the approximate curve by $\polygon{i} := \bigcup_{j=1}^{\vertexNum{i}}\edge{i}{j}$.
The discrete curve network is composed of the polygonal curves, i.e., $\polygon{} := \bigcup_{i=1}^{I_S}\polygon{i}$;
the ambient space $\R{2}$ is split into $I_R$ polygonal regions $\regionDiscrete{r}\,(r\in\naturalset{I_R})$, that is, $\R{2}\setminus\polygon{} = \bigcup_{r=1}^{I_R}\regionDiscrete{r}$.
If $\polygon{i}$ is open, then the vertex $\vertex{i}{1}$ (resp. $\vertex{i}{\vertexNum{i}}$) is supposed to be the start point (resp. endpoint) of the curve $\polygon{i}$.
Moreover, it is also supposed that the endpoints of an open curve in $\R{2}$ correspond to some triple junctions in the curve network $\polygon{}$.
Meanwhile, if $\Gamma_i$ is closed, then it is alternatively assumed that $\vertex{i}{0} = \vertex{i}{\vertexNum{i}}$ (see Figure~\ref{fig:open-closed-indexing}).

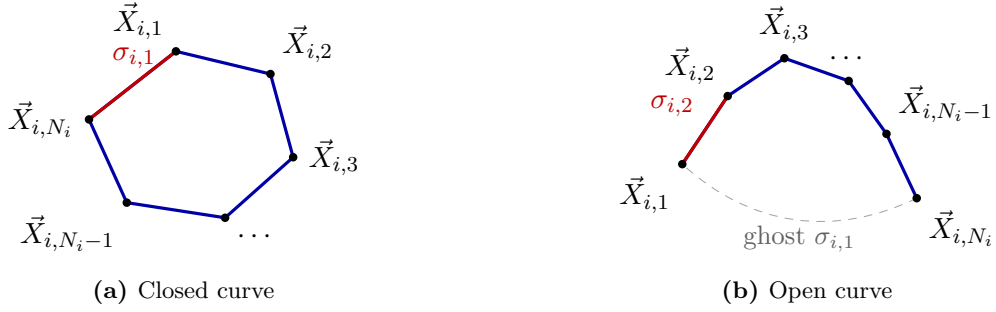
\begin{figure}[H]
\centering
\begin{subfigure}[t]{0.47\textwidth}
  \centering
  \begin{tikzpicture}[line cap=round,line join=round,>=Stealth]
    \coordinate (c1) at (0.00,1.55);
    \coordinate (c2) at (1.25,1.25);
    \coordinate (c3) at (1.55,0.15);
    \coordinate (c4) at (0.65,-0.65);
    \coordinate (c5) at (-0.65,-0.45);
    \coordinate (c6) at (-1.15,0.65);
    \draw[very thick,blue!65!black]
      (c1)--(c2)--(c3)--(c4)--(c5)--(c6)--cycle;
    \draw[very thick,red!75!black] (c6)--(c1);
    \foreach \p in {c1,c2,c3,c4,c5,c6}
      \fill[black] (\p) circle (1.6pt);
    \node[above left=1pt] at (c1) {$\vertex{i}{1}$};
    \node[above right=1pt] at (c2) {$\vertex{i}{2}$};
    \node[right=2pt] at (c3) {$\vertex{i}{3}$};
    \node[below right=1pt] at (c4) {$\cdots$};
    \node[below left=1pt] at (c5) {$\vertex{i}{N_i-1}$};
    \node[left=2pt] at (c6) {$\vertex{i}{N_i}$};
    \node[above=2pt,red!75!black] at ($(c6)!0.52!(c1)$)
      {$\edge{i}{1}$};
  \end{tikzpicture}
  \caption{Closed curve}
\end{subfigure}\hfill
\begin{subfigure}[t]{0.47\textwidth}
  \centering
  \begin{tikzpicture}[line cap=round,line join=round,>=Stealth]
    \coordinate (o1) at (-1.55,-0.25);
    \coordinate (o2) at (-0.95,0.65);
    \coordinate (o3) at (-0.20,1.15);
    \coordinate (o4) at (0.65,0.85);
    \coordinate (o5) at (1.15,0.15);
    \coordinate (o6) at (1.55,-0.70);
    \draw[very thick,blue!65!black]
      (o1)--(o2)--(o3)--(o4)--(o5)--(o6);
    \draw[very thick,red!75!black] (o1)--(o2);
    \draw[dashed,gray!70] (o6) to[bend left=34] (o1);
    \foreach \p in {o1,o2,o3,o4,o5,o6}
      \fill[black] (\p) circle (1.6pt);
    \node[below left=1pt] at (o1) {$\vertex{i}{1}$};
    \node[above left=1pt] at (o2) {$\vertex{i}{2}$};
    \node[above=2pt] at (o3) {$\vertex{i}{3}$};
    \node[above=2pt] at (o4) {$\cdots$};
    \node[above right=1pt] at (o5) {$\vertex{i}{N_i-1}$};
    \node[below right=1pt] at (o6) {$\vertex{i}{N_i}$};
    \node[above left=1pt,red!75!black] at ($(o1)!0.52!(o2)$)
      {$\edge{i}{2}$};
    \node[below=3pt,gray!75!black,font=\small] at (0,-0.82)
      {ghost $\edge{i}{1}$};
  \end{tikzpicture}
  \caption{Open curve}
\end{subfigure}
\caption{Vertex and edge indexing for closed and open polygonal curves $\polygon{i}$.}
\label{fig:open-closed-indexing}
\end{figure}

We also define a discrete variant of the normal vector field on $\polygon{i}$ by
\begin{equation*}
    \vertexNormal{i}{j} := \frac{\left(\vertex{i}{j} - \vertex{i}{j-1}\right)^\perp}{r_{i,j}}\qquad\mbox{for}\quad i\in\naturalset{I_S}\quad\mbox{and}\quad 1\leq j\leq \vertexNum{i},
\end{equation*}
where $r_{i,j} := |\edge{i}{j}| = \left|\vertex{i}{j} - \vertex{i}{j-1}\right|$, and for any vector $\bv{x}\in\mathbb{R}^2$, the symbol $\bv{x}^\perp$ denotes the vector which is obtained by rotating $\bv{x}$ by $90^\circ$ clockwise.
We let $\angleouter{i}{j}$ be the outer angle of $\polygon{i}$ at the vertex $\vertex{i}{j}$ and define a discrete variant of the curvature $\curvature$ by
\begin{equation*}
    \curvatureDisc{i}{j} := \frac{\tan{\left(\frac{\angleouter{i}{j-1}}{2}\right) + \tan{\left(\frac{\angleouter{i}{j}}{2}\right)}}}{r_{i,j}}\qquad\mbox{for}\quad i\in\naturalset{I_S}\quad\mbox{and}\quad 1\leq j\leq \vertexNum{i}.
\end{equation*}
Since $x\approx\tan{x}$ as $|x|\ll 1$, the quantity $\curvatureDisc{i}{j}$ approximates the turning rate of the curve at the edge center $\vertexCenter{i}{j}$ defined by
\begin{equation*}
    \vertexCenter{i}{j} := \frac{\vertex{i}{j-1} + \vertex{i}{j}}{2}\qquad\mbox{for}\quad i\in\naturalset{I_S}\quad\mbox{and}\quad 1\leq j\leq \vertexNum{i},
\end{equation*}
provided that $\vertexNum{i}$ is taken so large that $r_{i,j}\ll 1$, and hence $\curvatureDisc{i}{j}\approx\tfrac{d\phi}{ds} = -\curvature$, where $s$ denotes the arc-length parameter on $\Gamma_i$. Here the sign is opposite to the normal-direction curvature $\curvature$ of Section~2: the outer angle $\angleouter{i}{j}$ is positive when $\polygon{i}$ turns counter-clockwise, so $\curvatureDisc{i}{j} > 0$ for a curve that is convex toward the region $\regionDiscrete{\regionindexm{i}}$, whereas $\curvature < 0$ there. We retain $\curvatureDisc{i}{j}$ (the implemented quantity) as the right-hand side of the discrete Gibbs--Thomson rows below and absorb the sign into the jump ordering accordingly.
Here, we stress that the index $j$ of the vertices should start from $1$ if $\polygon{i}$ is open, and it should start from $0$ if $\polygon{i}$ is closed
since we cannot define the outer angle at the endpoints of open curves.
\begin{rem}
In the previous work \cite{E24}, the author considered open curves whose endpoints are not triple junctions and lie on the boundary of a half space.
Moreover, he defined imaginary vertices to define the values of $\curvatureDisc{i}{j}$ on the boundary of the half space.
However, in the present model, triple junctions can evolve in time, and the curvature at these points does not make sense.
The half-space setting is revisited in \Section{sec:halfspace}, where the wall endpoints become mobile contacts that meet the wall orthogonally and may coexist with evolving triple junctions.
\end{rem}
\begin{rem}[Curvature correction near triple junctions]\label{rem:dcorr}
For an open curve $\polygon{i}$, the standard formula for $\curvatureDisc{i}{j}$ uses the outer angles at both endpoints of the edge $\edge{i}{j}$, which are unreliable for the four edges adjacent to the triple junctions ($j = 2, 3, N-1, N$ with $N := \vertexNum{i}$), since the outer angle is undefined at the junction vertices. By default the implementation re-evaluates the discrete curvature only at these four edges and only for the right-hand side of the Gibbs--Thomson rows \eqref{eq:GTLS}, using one-sided differences of the turning angle that do not involve the junction vertex,
\begin{align*}
    &\curvatureDisc{i}{2} \leftarrow \frac{\angleouter{i}{2}}{(\ell_2+\ell_3)/2}, \quad
    \curvatureDisc{i}{3} \leftarrow \frac{\angleouter{i}{2}+\angleouter{i}{3}}{\ell_2/2+\ell_3+\ell_4/2},\\
    &\curvatureDisc{i}{N} \leftarrow \frac{\angleouter{i}{N-1}}{(\ell_{N-1}+\ell_N)/2}, \quad
    \curvatureDisc{i}{N-1} \leftarrow \frac{\angleouter{i}{N-2}+\angleouter{i}{N-1}}{\ell_{N-2}/2+\ell_{N-1}+\ell_N/2},
\end{align*}
with $\ell_j := r_{i,j}$. This correction applies to open curves only; for closed curves the discrete curvature is used unchanged. It modifies neither the curvature stored for any other purpose nor the closed-curve rows, and it can be disabled, recovering the uncorrected scheme. In numerical tests on a four-phase open-junction network, the uncorrected junction-edge curvature stays about a factor of two below the corrected one-sided estimate uniformly in the resolution, a deficit that does not improve under refinement; the switch is confined to the field accuracy near the junctions and does not reach the conserved quantities, which are protected by the projection (Proposition~\ref{prop:velocity-level}) and the geometric angle restoration independently of it.
\end{rem}

\section{Fully discrete scheme}
\label{sec:fully-discrete-scheme}
In this section, we give a fully discrete scheme to approximate the interface evolution governed by \eqref{system:all}.
To this end, given a smooth curve network $\Gamma$, we first obtain the spatial discretization $\polygon{}$ of $\Gamma$ using the method introduced in Section~\ref{sec:spatial_discretization}.
After that, we compute the normal velocity $V$ of $\polygon{}$ at each vertex, which is determined by \eqref{system:motion}.
Therefore, we are led to compute $\bvp{w}$ at each time step.
In this paper, we adopt the charge simulation method (CSM) to solve the Dirichlet boundary problem \eqref{system:laplace} and \eqref{system:gibbs-thomson}.
\fix{Figure~\ref{fig:algorithm} at the end of this section collects the whole step in one diagram.}
We now explain a basic idea of the CSM.

Let $\fundSolution$ be the fundamental solution to the Laplace equation in $\mathbb{R}^2$, that is
\begin{equation*}
    \fundSolution(\bv{x}) := \frac{1}{2\pi}\log|\bv{x}| \qquad\mbox{for}\quad \bv{x}\in\mathbb{R}^2\setminus\{\bv{0}\}.
\end{equation*}
Then, for $r\in\naturalset{I_R}$ and $p\in\naturalset{I_P}$, the approximate solution is defined by
\begin{equation}\label{eq:appsol}
    \smallChemical{r}{p}(\bv{x}) := c^{(r)}_{p} + \sum_{i\in \curveRegionMap^{-1}(r)}\sum_{j=1}^{\vertexNum{i}}\charge{r}{p}{i,j}\left\{\fundSolution\left(\bv{x} - \singularPoint{r}{i}{j}\right) - \fundSolution\left(\bv{x} - \singularPointDummy{r}{i}{j}\right)\right\}
    \qquad\mbox{for}\quad \bv{x}\in\mathbb{R}^2,
\end{equation}
where we recall that $\smallChemical{r}{p} := w_p\lfloor_{\regionDiscrete{r}}$, the constant $c^{(r)}_p$ is a region-wise additive constant (one per region $r$ and phase $p$, not shared across regions), and $\singularPoint{r}{i}{j}$ and $\singularPointDummy{r}{i}{j}$ are \textit{the charge points} of the CSM, both placed outside $\domain{r}$.

\begin{rem}
    The origin of construction for approximate solutions in \eqref{eq:appsol} goes back to Murota \cite{M95}
    in which an invariant structure against affine transformations of the coordinate has been invoked in the definition of approximate solutions.
    Due to the appearance of one more unknown, the zero-sum condition on the unknowns has been imposed to solve a linear system.
    Therein, the dummy singular points $\singularPointDummy{r}{i}{j}$ did not indeed appear in the representation of approximate solutions,
    and the use of this kind of points has been proposed by Sakakibara and Yazaki \cite{SY19}.
    In their study, an area-preserving property has been encoded into a linear system instead of Murota's zero-sum condition.
    This structure ensures invariance property of the approximate solution with respect to scale transformation (see \cite[Remark 1]{E24}).
    We note that this invariant property is valid only in the planar case thanks to the logarithm of the fundamental solution.
\end{rem}

For each $r\in\naturalset{I_R}$, we let
\[
    M_r := \sum_{i\in \curveRegionMap^{-1}(r)}\#\left\{\,\mbox{edges of }\polygon{i}\,\right\}
\]
be the total number of collocation points (edge centers) on $\partial\regionDiscrete{r}$, and let
\begin{equation}
    \label{eq:orientation}
    \bv{\nu}^{\mathrm{out}}_{r,i,j} := o_{r,i}\,\vertexNormal{i}{j}, \qquad o_{r,i} := \begin{cases} +1 & r = \regionindexm{i},\\ -1 & r = \regionindexp{i},\end{cases}
\end{equation}
denote the unit normal of $\polygon{i}$ pointing out of $\regionDiscrete{r}$ (note that $\vertexNormal{i}{j}$ points from $\regionDiscrete{\regionindexm{i}}$ to $\regionDiscrete{\regionindexp{i}}$). The charge points are then set as
\begin{equation}\label{eq:charge}
    \singularPoint{r}{i}{j} := \vertexCenter{i}{j} + \frac{1}{\sqrt{M_r}}\,\bv{\nu}^{\mathrm{out}}_{r,i,j}\qquad\mbox{and}\qquad\singularPointDummy{r}{i}{j} := \vertexCenter{i}{j} + M_r^{\,\beta}\,\bv{\nu}^{\mathrm{out}}_{r,i,j},\qquad \beta := \tfrac{3}{2},
\end{equation}
where the constant $\beta$ has been chosen according to \cite[\S 4]{E24}.
Both points lie on the same (outer) side of $\regionDiscrete{r}$: the principal point $\singularPoint{r}{i}{j}$ at the short distance $1/\sqrt{M_r}$, and the auxiliary point $\singularPointDummy{r}{i}{j}$ far away at distance $M_r^{\,\beta}$.
It was shown in \cite[Lemma 4]{E24} that the approximate solution \eqref{eq:appsol} satisfies the Neumann boundary condition \eqref{system:neumann}, i.e.,
$$\nabla\fundSolution\left(\bv{x}-\singularPoint{r}{i}{j}\right)-\nabla\fundSolution\left(\bv{x}-\singularPointDummy{r}{i}{j}\right) = O\left(\frac{1}{|\bv{x}|^2}\right)\qquad\text{as}\quad |\bv{x}|\to\infty,$$
matching the far-field condition in \eqref{system:neumann}.
When $\partial\regionDiscrete{r}$ has several connected components (for instance the exterior region of a configuration containing an isolated closed curve), the points \eqref{eq:charge} are placed component by component.
Figure~\ref{fig:csm-schematic} illustrates this charge-simulation set-up near an interface curve $\polygon{i}$.

{\color{revision}
The principal distance in \eqref{eq:charge} is a compromise. The point must stand off far
enough that the kernel is not nearly singular at the collocation points, yet not so far that
it crosses the next interface; $1/\sqrt{M_r}$ lies between the edge length $O(1/M_r)$ and the
size of the configuration. The auxiliary distance has only to be large, which $M_r^{\beta}$
ensures. How much the computed results depend on this choice is measured in
\Section{subsec:convergence}.
}

\begin{figure}[H]
  \centering
  \begin{tikzpicture}[>={Stealth[length=2.2mm]}, line join=round, scale=1.15]
    \definecolor{regL}{RGB}{86,180,233}   
    \definecolor{regR}{RGB}{230,159,0}    
    \coordinate (V1) at (0.46, 2.20);
    \coordinate (V2) at (0.08, 1.30);
    \coordinate (V3) at (-0.20, 0.64);
    \coordinate (V4) at (-0.20,-0.64);
    \coordinate (V5) at (0.08,-1.30);
    \coordinate (V6) at (0.46,-2.20);
    \fill[regL!13] (V1)--(V2)--(V3)--(V4)--(V5)--(V6)--(-3.0,-2.45)--(-3.0,2.45)--cycle;
    \fill[regR!13] (V1)--(V2)--(V3)--(V4)--(V5)--(V6)--(3.0,-2.45)--(3.0,2.45)--cycle;
    \coordinate (M) at ($(V3)!0.5!(V4)$);
    \coordinate (yL) at ($(M)+(-1.35,0)$);
    \coordinate (zL) at ($(M)+(-2.28,0)$);
    \coordinate (yR) at ($(M)+(1.35,0)$);
    \coordinate (zR) at ($(M)+(2.28,0)$);
    \draw[gray!60,dotted,thick] (M) -- (zL);
    \draw[gray!60,dotted,thick] (M) -- (zR);
    \draw[very thick] (V1)--(V2)--(V3)--(V4)--(V5)--(V6);
    \foreach \p in {V1,V2,V3,V4,V5,V6}{\fill (\p) circle (1.6pt);}
    \foreach \a/\b in {V1/V2,V2/V3,V4/V5,V5/V6}{
      \coordinate (m) at ($(\a)!0.5!(\b)$);
      \draw[gray!55] (m)+(-1.6pt,-1.6pt) -- +(1.6pt,1.6pt);
      \draw[gray!55] (m)+(-1.6pt,1.6pt) -- +(1.6pt,-1.6pt);
    }
    \draw (M)+(-1.9pt,-1.9pt) -- +(1.9pt,1.9pt);
    \draw (M)+(-1.9pt,1.9pt) -- +(1.9pt,-1.9pt);
    \foreach \p in {yL,zL,yR,zR}{\fill[white] (\p) circle (1.9pt); \draw (\p) circle (1.9pt);}
    \draw[->,black,thick] (M) -- ++(-0.55,0) node[above,black,inner sep=1pt] {$\vertexNormal{i}{j}$};
    \node[right] at ($(V3)+(0.10,0.18)$)  {$\vertex{i}{j}$};
    \node[right] at ($(V4)+(0.10,-0.18)$) {$\vertex{i}{j-1}$};
    \node[right] at ($(M)+(0.12,-0.17)$)  {$\vertexCenter{i}{j}$};
    \node[below] at (yL) {$\singularPoint{\regionindexm{i}}{i}{j}$};
    \node[below] at (zL) {$\singularPointDummy{\regionindexm{i}}{i}{j}$};
    \node[below] at (yR) {$\singularPoint{\regionindexp{i}}{i}{j}$};
    \node[below] at (zR) {$\singularPointDummy{\regionindexp{i}}{i}{j}$};
    \node at (-2.5, 1.9) {$\regionDiscrete{\regionindexp{i}}$};
    \node at (2.5, 1.9)  {$\regionDiscrete{\regionindexm{i}}$};
    \node[below right] at (V6) {$\polygon{i}$};
  \end{tikzpicture}
  \caption{Charge-simulation set-up near a polygonal curve $\polygon{i}$.}
  \label{fig:csm-schematic}
\end{figure}
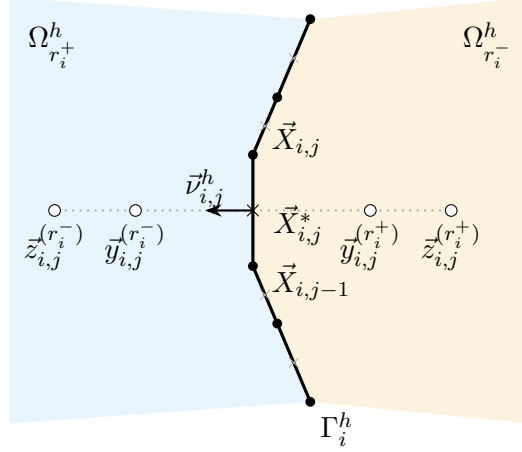
Using these notations, we now explain how to determine the coefficients $c^{(r)}_{p}$ and $\charge{r}{p}{i,j}$.
We note that the chemical potentials $\smallChemical{r}{p}$ sum up to zero for each $r\in \naturalset{I_R}$,
so we only need to determine the coefficients for $1 \leq p \leq I_P - 1$.
Since the additive constant $c^{(r)}_p$ is region-wise, each region contributes one constant per phase. Thus, the number of unknowns is equal to
\begin{equation}
    \label{eq:num-unknowns}
    \unknownsNum := (I_P - 1)\left\{\sum_{r=1}^{I_R}\left(\sum_{i\in \curveRegionMap^{-1}(r)}N_i + 1\right)\right\}.
\end{equation}
First, the Gibbs--Thomson law \eqref{system:gibbs-thomson} is imposed on both sides of each curve. Recalling the discrete curvature $\curvatureDisc{i}{j}$ together with its sign convention fixed in Section~\ref{sec:spatial_discretization}, the relation to be discretized reads
\begin{equation}\label{eq:GTL}
    \smallChemical{\regionindexpm{i}}{\phaseindexm{i}} - \smallChemical{\regionindexpm{i}}{\phaseindexp{i}} = \tension{i}\,\kappa^h_i\quad\mbox{on}\quad\polygon{i}\quad\mbox{for}\quad i\in\naturalset{I_S}.
\end{equation}
\noindent
Second, the continuity condition of the chemical potential across the curve network $\Gamma$ \eqref{system:continuity} is given by
\begin{equation}\label{eq:CC}
    \smallChemical{\regionindexp{i}}{p} = \smallChemical{\regionindexm{i}}{p}\quad\mbox{on}\quad\polygon{i}\quad\mbox{for}\quad i\in\naturalset{I_S}\quad\mbox{and}\quad 1\leq p\leq I_P - 1.
\end{equation}

We can directly calculate the normal derivative of the approximate solution \eqref{eq:appsol} as follows:
\begin{equation*}
    \nabla\smallChemical{r}{p}(\bv{x}) = \sum_{i\in \curveRegionMap^{-1}(r)}\sum_{j=1}^{\vertexNum{i}}\charge{r}{p}{i,j}\left\{\nabla\fundSolution\left(\bv{x} - \singularPoint{r}{i}{j}\right) - \nabla\fundSolution\left(\bv{x} - \singularPointDummy{r}{i}{j}\right)\right\},
\end{equation*}
and the normal velocity $v_{i,j}$ at the edge center $\vertexCenter{i}{j}$ is computed from the motion law \eqref{system:motion},
which gives the two equivalent representations:
\begin{equation}\label{eq:Vij}
    \left(\nabla\smallChemical{\regionindexp{i}}{\phaseindexp{i}} - \nabla\smallChemical{\regionindexm{i}}{\phaseindexp{i}}\right)\cdot\vertexNormal{i}{j} = -v_{i,j}\quad\mbox{and}\quad
    \left(\nabla\smallChemical{\regionindexp{i}}{\phaseindexm{i}} - \nabla\smallChemical{\regionindexm{i}}{\phaseindexm{i}}\right)\cdot\vertexNormal{i}{j} = v_{i,j}.
\end{equation}
In the implementation, we choose the second representation in \eqref{eq:Vij}, i.e., from the phase $\phaseindexm{i}$ ($= \regionPhaseMap(\regionindexm{i})$) on both sides of $\polygon{i}$.

We need $(I_P - 1)$ more equations to determine the coefficients.
To this end, we take the area-preserving condition for each phase into account (see Proposition~\ref{prop:ap}).
Namely, we require that
\begin{equation*}
    0 = \frac{d}{dt}\sum_{r\in \regionPhaseMap^{-1}(p)}\left|\domain{r}(t)\right| = -\sum_{r\in \regionPhaseMap^{-1}(p)}\int_{\partial\domain{r}}V_{in}(r)\,\dH{1}\quad\mbox{for}\quad 1\leq p \leq I_P - 1,
\end{equation*}
where $V_{in}(r)$ denotes the inward normal velocity at the boundary $\partial\domain{r}$.
We discretize the integration over the boundary $\partial\domain{r}$ by the sum of the integrals over the edges $\edge{i}{j}$.
Then, we obtain
\begin{equation}\label{eq:AP}
    \sum_{r\in \regionPhaseMap^{-1}(p)}\left(\sum_{\substack{i\in \curveRegionMap^{-1}(r)\\ r = \regionindexp{i}}}\sum_{j=1}^{\vertexNum{i}}v_{i,j}r_{i,j} + \sum_{\substack{i\in \curveRegionMap^{-1}(r)\\ r = \regionindexm{i}}}\sum_{j=1}^{\vertexNum{i}}-v_{i,j}r_{i,j}\right) = 0\quad\mbox{for}\quad 1\leq p\leq I_P - 1.
\end{equation}
The condition \eqref{eq:AP} constrains only the change of area.
The area surrounded by the oriented boundary circuit $\partial\regionDiscrete{r}$ is computed by the \textit{shoelace formula}:
\begin{equation}\label{eq:shoelace-area}
    \left|\regionDiscrete{r}\right|
    = \frac{1}{2}\left|\,
        \sum_{i\in \curveRegionMap^{-1}(r)} o_{r,i}
        \sum_{j=1}^{\vertexNum{i}} \vertex{i}{j-1}\times\vertex{i}{j}
    \,\right|,
    \qquad
    \bv{a}\times\bv{b} := a_1 b_2 - a_2 b_1 ,
\end{equation}
where the orientation sign $o_{r,i}$ of \eqref{eq:orientation} traverses each
incident curve in the direction consistent with $\partial\regionDiscrete{r}$; adjacent
curves share their triple-junction endpoints, so the circuit closes with no
connecting edges. Differentiating \eqref{eq:shoelace-area} in time reproduces,
to leading order, the left-hand side of \eqref{eq:AP}: the linear system enforces
$\tfrac{d}{dt}|\domain{r}(t)|=0$ at the velocity level, while
\eqref{eq:shoelace-area} is the polygonal area whose residual drift $E_A$ is
reported in Section~\ref{sec:experiments}.
Here, we note that the inward normal velocity $V_{in}(r)$ at the boundary $\partial\domain{r}$ is given by
\begin{equation*}
    V_{in}(r) = \begin{cases}
        v_{i,j}\qquad&\mbox{for}\quad r = \regionindexp{i},\\
        -v_{i,j}\qquad&\mbox{for}\quad r = \regionindexm{i}.
    \end{cases}
\end{equation*}
To compute the above $v_{i,j}$, we can use the formulae \eqref{eq:Vij}.

Finally, we represent the formulae \eqref{eq:GTL}, \eqref{eq:CC}, and \eqref{eq:AP} as a linear system of the coefficients $c^{(r)}_{p}$ and $\charge{r}{p}{i,j}$.
To this end, we introduce the following notations:
\begin{align*}
    \scalarG{r}{c_1}{\ell}{c_2}{j} &:= \fundSolution\left(\vertexCenter{c_1}{\ell} - \singularPoint{r}{c_2}{j}\right) - \fundSolution\left(\vertexCenter{c_1}{\ell} - \singularPointDummy{r}{c_2}{j}\right),\\
    \vectorH{r}{c_1}{\ell}{c_2}{j} &:= \nabla\fundSolution\left(\vertexCenter{c_1}{\ell} - \singularPoint{r}{c_2}{j}\right) - \nabla\fundSolution\left(\vertexCenter{c_1}{\ell} - \singularPointDummy{r}{c_2}{j}\right),
\end{align*}
for $r\in\naturalset{I_R}$, $c_1,\,c_2\in\naturalset{I_S}$, $\ell\in\naturalset{\vertexNum{c_1}}$, and $j\in\naturalset{\vertexNum{c_2}}$.
\medskip

Then, we obtain the following linear system.
\medskip

\noindent
\begin{subequations}\label{eq:LS}
    \textbf{[Gibbs--Thomson law]} For every $i\in\naturalset{I_S}$, $k\in\naturalset{\vertexNum{i}}$,
    and each side $\regionindexpm{i}\in\{\regionindexm{i},\regionindexp{i}\}$, it holds that
    \begin{equation}\label{eq:GTLS}
        c^{(\regionindexpm{i})}_{\phaseindexm{i}} - c^{(\regionindexpm{i})}_{\phaseindexp{i}} + \sum_{\ell\in \curveRegionMap^{-1}(\regionindexpm{i})}\sum_{j=1}^{\vertexNum{\ell}}\left(\charge{\regionindexpm{i}}{\phaseindexm{i}}{\ell,j} - \charge{\regionindexpm{i}}{\phaseindexp{i}}{\ell,j}\right) \scalarG{\regionindexpm{i}}{i}{k}{\ell}{j} = \tension{i}\curvatureDisc{i}{k}.
    \end{equation}
    \textbf{[Continuity condition]} For every $i\in\naturalset{I_S}$, $k\in\naturalset{\vertexNum{i}}$, and $1\leq p\leq I_P - 1$, it holds that
    \begin{equation}\label{eq:CCLS}
        c^{(\regionindexp{i})}_{p} - c^{(\regionindexm{i})}_{p} + \sum_{\ell\in \curveRegionMap^{-1}(\regionindexp{i})}\sum_{j=1}^{\vertexNum{\ell}}\charge{\regionindexp{i}}{p}{\ell,j}\scalarG{\regionindexp{i}}{i}{k}{\ell}{j} - \sum_{\ell\in \curveRegionMap^{-1}(\regionindexm{i})}\sum_{j=1}^{\vertexNum{\ell}}\charge{\regionindexm{i}}{p}{\ell,j}\scalarG{\regionindexm{i}}{i}{k}{\ell}{j} = 0.
    \end{equation}
    \textbf{[Area-preserving condition]} For every $1\leq p\leq I_P-1$, it holds that
    \begin{equation}\label{eq:APLS}
        \sum_{r\in \regionPhaseMap^{-1}(p)}\left(\sum_{\substack{i\in \curveRegionMap^{-1}(r)\\ r = \regionindexp{i}}}\sum_{j=1}^{\vertexNum{i}} r_{i,j}\, v_{i,j} - \sum_{\substack{i\in \curveRegionMap^{-1}(r)\\ r = \regionindexm{i}}}\sum_{j=1}^{\vertexNum{i}} r_{i,j}\, v_{i,j}\right) = 0,
    \end{equation}
    where the single normal velocity $v_{i,j}$ is taken in the $\phaseindexm{i}$-representation (the second identity of \eqref{eq:Vij}):
    \begin{equation*}
        v_{i,j} =  \left(\sum_{\ell\in \curveRegionMap^{-1}(\regionindexp{i})}\sum_{k=1}^{\vertexNum{\ell}}\charge{\regionindexp{i}}{\phaseindexm{i}}{\ell,k}\vectorH{\regionindexp{i}}{i}{j}{\ell}{k} - \sum_{\ell\in \curveRegionMap^{-1}(\regionindexm{i})}\sum_{k=1}^{\vertexNum{\ell}}\charge{\regionindexm{i}}{\phaseindexm{i}}{\ell,k}\vectorH{\regionindexm{i}}{i}{j}{\ell}{k}\right) \cdot\vertexNormal{i}{j}.
    \end{equation*}
\end{subequations}

The rows \eqref{eq:GTLS}, \eqref{eq:CCLS}, and \eqref{eq:APLS} assemble into a linear system $A\bv{q} = \bv{b}$ for the unknown vector $\bv{q}\in\R{\unknownsNum}$
collecting the region-wise coefficients $\charge{r}{p}{i,j}$ and constants $c^{(r)}_p$ ($1\le p\le I_P-1$).

Let $C\bv{q} = \bv{0}$ denote the $(I_P-1)$ area-preserving rows \eqref{eq:APLS}, and let $A_1\bv{q} = \bv{b}_1$ denote the remaining Gibbs--Thomson \eqref{eq:GTLS} and continuity \eqref{eq:CCLS} rows.
We enforce $C\bv{q}=\bv{0}$ as a hard constraint by the orthogonal projection onto its null space:
\begin{equation}\label{eq:proj}
    P := I - C^\top\left(CC^\top\right)^{-1}C,
\end{equation}
and we solve the projected least-squares problem:
\begin{equation}\label{eq:pcgls}
    \min_{\bv{y}\in\R{\unknownsNum}} \left\|A_1 P\bv{y} - \bv{b}_1\right\|_2^2, \qquad \bv{q} := P\bv{y},
\end{equation}
by the conjugate-gradient least-squares (CGLS) iteration. By construction, $C\bv{q} = CP\bv{y} = \bv{0}$ holds to machine precision,
and thus the discrete area-preserving condition \eqref{eq:APLS} is satisfied at the level of the reconstructed velocity field independently of the residual of \eqref{eq:pcgls}. The Gibbs--Thomson and continuity rows are then satisfied in the least-squares sense.

The normal velocities $v_{i,j}$ are computed at the edge centers $\vertexCenter{i}{j}$, whereas the curve is advanced by moving its vertices $\vertex{i}{j}$. We therefore transfer the edge data to the vertices by averaging the adjacent edge normals and tangents. Define the (unnormalized) vertex normal and tangent
\begin{equation*}
    \overline{\bv{\nu}}_{i,j} := \frac{\vertexNormal{i}{j} + \vertexNormal{i}{j+1}}{2\cos\left(\frac{\angleouter{i}{j}}{2}\right)}, \qquad
    \overline{\bv{\tau}}_{i,j} := \frac{\bv{\tau}_{i,j} + \bv{\tau}_{i,j+1}}{2\cos\left(\frac{\angleouter{i}{j}}{2}\right)},
\end{equation*}
where $\bv{\tau}_{i,j}$ is the unit tangent of the edge $\edge{i}{j}$. Each vertex is then advanced by the normal velocity along $\overline{\bv{\nu}}_{i,j}$ together with a tangential redistribution velocity $T_{i,j}$,
which will be determined by the \textit{uniform distribution method} (UDM) below, along $\overline{\bv{\tau}}_{i,j}$:
\begin{equation}\label{eq:evolve}
    \vertexSeq{n+1}{i}{j} = \vertexSeq{n}{i}{j} + \Delta t V^{(n)}\qquad\text{with}\quad V^{(n)} := v_{i,j}\,\overline{\bv{\nu}}_{i,j} + T_{i,j}\,\overline{\bv{\tau}}_{i,j}.
\end{equation}
For an open curve the two endpoints ($j=1$ and $j=\vertexNum{i}$) are held fixed at this stage (their velocity is set to zero) and are subsequently relocated by the triple-junction correction below; for a closed curve all vertices are advanced by \eqref{eq:evolve}. Here, $\Delta t > 0$ is a time step size;
in the numerical experiments, we take $\Delta t = c / N^2$ with a fixed constant $c$ (typically $c\in[0.01, 0.5]$), where $N$ is the number of vertices per curve.
\begin{rem}\label{rem:constants}
The region-wise constants $c^{(r)}_p$ are not shared across regions.
Their differences across each curve are fixed by the continuity rows \eqref{eq:CCLS},
and the remaining gauge freedom (a global additive constant per phase) is harmless for the velocity reconstruction,
which depends only on the gradients $\nabla\smallChemical{r}{p}$. The constrained least-squares \eqref{eq:pcgls} selects a representative consistent with all rows.
\end{rem}
To avoid mesh degeneration, we redistribute the vertices along the curve by the tangential velocity $T_{i,j}$ of \eqref{eq:evolve}; since this motion is tangential, it does not change the geometric evolution.
To this end, we follow the strategy of \cite[Eq.(21)]{E24}, although we
modify it so that each curve may be open.
Fix a polygonal curve $\polygon{i}$ with $N := \vertexNum{i}$ edges,
write $\ell_j := r_{i,j}$ for the edge lengths, $\polygonLength{} := \sum_{j=1}^{N_i} \ell_j$ for its length, and
\begin{equation*}
    \ldotDiscrete
    := \sum_{j} \curvatureDisc{i}{j}\, v_{i,j}\,\ell_j
    \approx \int_\Gamma \curvature V\,\dH{1} = \ldot
\end{equation*}
for the discrete derivative of length, where the sum runs over the real edges
($j=2,\dots,N$ for an open curve, $j=1,\dots,N$ for a closed one).
The redistribution drives the edge lengths toward the uniform value $L/N$ (resp.\ $L/(N-1)$ for an open curve),
where $L$ denotes the length of the curve which is supposed to be approximated by the polygonal curve $\polygon{i}$.
We briefly explain its procedure. For a closed curve, we adopt the same formula in \cite[Eq.(21)]{E24}.
For an open curve, the endpoint tangential velocities are fixed, i.e., $T_{i,1} = T_{i,N} = 0$,
and the interior values solve the bidiagonal least-squares system:
\begin{multline*}
    -\cos\!\tfrac{\angleouter{i}{j}}{2}\,T_{i,j} + \cos\!\tfrac{\angleouter{i}{j+1}}{2}\,T_{i,j+1}\\
    = \frac{\ldotDiscrete}{N-1} + \left(\frac{\lDiscrete}{N-1} - \ell_{j+1}\right)10 N - v_{i,j}\sin\!\tfrac{\angleouter{i}{j+1}}{2} - v_{i,j-1}\sin\!\tfrac{\angleouter{i}{j}}{2}
\end{multline*}
for $2\leq j\leq N-1$, which we solve by CGLS. In both cases the factor $10N$ is an implementation parameter controlling the strength of the equi-distribution.

The triple junctions are updated geometrically, not as rows of the linear system \eqref{eq:GTLS}, \eqref{eq:CCLS}, and \eqref{eq:APLS}.
During the free curve update, the endpoints of each open curve are held fixed, while all interior vertices are advanced by \eqref{eq:evolve}. For each $k\in\naturalset{I_T}$, let $\bv{a}_1,\bv{a}_2,\bv{a}_3$ denote the updated interior vertices adjacent to the junction $\triplejunction{k}$ on its three incident curves. We compute a new common junction position $\bv{p}^*_k$ such that the directions $\bv{a}_m-\bv{p}^*_k\,(m=1,2,3)$ pairwise enclose the equilibrium angles ($120^\circ$ in the equal-tension case). The point $\bv{p}^*_k$ is obtained by a damped, regularized Gauss--Newton iteration applied to the three angle-cosine residuals, initialized at the centroid of $\bv{a}_1,\bv{a}_2,\bv{a}_3$.
The corresponding endpoint of every incident curve is then reset to this same point (Figure~\ref{fig:tj-correction}). Finally, the geometric quantities (edges, normals, outer angles, and the discrete curvature) of the affected curves are recomputed from the corrected vertices according to the procedure explained in Section~\ref{sec:spatial_discretization}.
Since closed curves have no triple-junction endpoints, they are advanced by \eqref{eq:evolve} with all of their vertices updated, without any endpoint correction.
We stress that the Herring--Young condition is enforced here as a geometric correction.

\begin{figure}[H]
  \centering
  \begin{tikzpicture}[>={Stealth[length=2.2mm]}, line join=round, scale=1.05]
    \coordinate (A1) at (90:1.35);
    \coordinate (A2) at (210:1.25);
    \coordinate (A3) at (330:1.45);
    \coordinate (Pold) at (0.18,-0.18);
    \draw[gray!55] (A1) -- ++(80:0.50) -- ++(60:0.50);
    \draw[gray!55] (A2) -- ++(215:0.45) -- ++(240:0.45);
    \draw[gray!55] (A3) -- ++(320:0.50) -- ++(300:0.50);
    \draw[very thick] (A1) -- (Pold);
    \draw[very thick] (A2) -- (Pold);
    \draw[very thick] (A3) -- (Pold);
    \foreach \a in {A1,A2,A3}{\fill (\a) circle (1.6pt);}
    \fill (Pold) circle (1.8pt);
    \node[right,inner sep=3pt] at (A1) {$\bv{a}_1$};
    \node[below left,inner sep=2pt] at (A2) {$\bv{a}_2$};
    \node[below,inner sep=3pt] at (A3) {$\bv{a}_3$};
    \node[below,inner sep=5pt] at (Pold) {$\bv{p}^{(n)}_k$};
    \node at (0,-2.15) {(a)};
    \begin{scope}[shift={(5.6,0)}]
      \coordinate (P) at (0,0);
      \coordinate (B1) at (90:1.35);
      \coordinate (B2) at (210:1.25);
      \coordinate (B3) at (330:1.45);
      \draw[gray!55] (B1) -- ++(80:0.50) -- ++(60:0.50);
      \draw[gray!55] (B2) -- ++(215:0.45) -- ++(240:0.45);
      \draw[gray!55] (B3) -- ++(320:0.50) -- ++(300:0.50);
      \foreach \b in {B1,B2,B3}{\draw[very thick] (\b) -- (P);}
      \draw[gray!80] (90:0.40) arc[start angle=90, end angle=210, radius=0.40];
      \node[gray!50!black] at (150:0.70) {$120^\circ$};
      \draw[gray!80] (330:0.34) arc[start angle=330, end angle=450, radius=0.34];
      \node[gray!50!black] at (30:0.66) {$120^\circ$};
      \foreach \b in {B1,B2,B3}{\fill (\b) circle (1.6pt);}
      \fill (P) circle (1.8pt);
      \node[below,inner sep=5pt] at (P) {$\bv{p}^*_k$};
      \node[right,inner sep=3pt] at (B1) {$\bv{a}_1$};
      \node[below left,inner sep=2pt] at (B2) {$\bv{a}_2$};
      \node[below,inner sep=3pt] at (B3) {$\bv{a}_3$};
      \node at (0,-2.15) {(b)};
    \end{scope}
  \end{tikzpicture}
  \caption{Post-hoc geometric triple-junction update: (a) after the free curve update;
    (b) \fix{r}eset of the endpoints to the
    new common point $\bv{p}^*_k$.}
  \label{fig:tj-correction}
\end{figure}
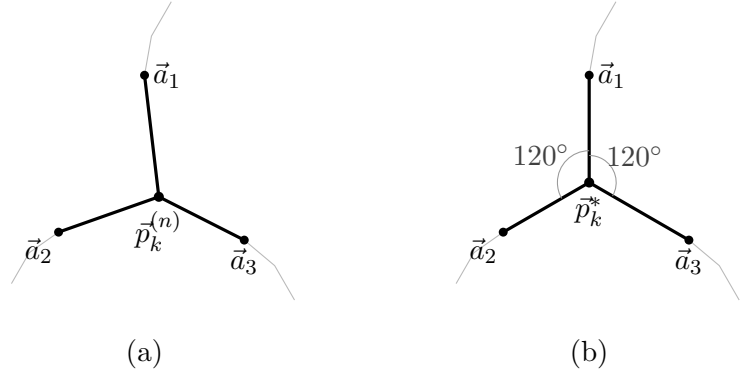

\begin{figure}[H]
  \centering
  {\color{revision}
  \begin{tikzpicture}[
      >={Stealth[length=2.2mm]},
      box/.style={draw, rounded corners=2pt, align=left, inner sep=3pt,
                  text width=0.52\linewidth, font=\footnotesize},
      hs/.style={draw, rounded corners=2pt, align=left, inner sep=3pt,
                 text width=0.30\linewidth, font=\footnotesize, anchor=north east},
      arr/.style={->, thick}
    ]
    \node[box] (geo) {\textbf{Polygonal network $\polygon{}$ at $t_n$}: edges, normals,
      outer angles and the discrete curvature $\curvatureDisc{i}{j}$
      (\Section{sec:spatial_discretization}), re-evaluated next to a junction
      (\Remark{rem:dcorr})};
    \node[box, below=5mm of geo] (chgW) {\textbf{Charge points}: a principal and an
      auxiliary point per collocation point \eqref{eq:charge}};
    \node[hs] (chgH) at ($(chgW.north west)+(-6mm,0)$) {\textbf{Charge points}: the principal one at
      the local offset $\alpha\,r_{i,j}$, every one reflected in the wall
      \eqref{eq:hs-source}};
    \path let \p1=(chgW.south), \p2=(chgH.south) in
      coordinate (mrg1) at (\x1,{min(\y1,\y2)-2.5mm})
      node[box, anchor=north] (asm) at (\x1,{min(\y1,\y2)-5mm})
        {\textbf{Collocation rows}: Gibbs--Thomson \eqref{eq:GTLS}, continuity
         \eqref{eq:CCLS} and phase areas \eqref{eq:APLS}};
    \node[box, below=5mm of asm] (solW) {\textbf{Projected solve}: CGLS \eqref{eq:pcgls};
      every iterate lies in $\ker C$, so the area rows hold whatever the residual};
    \node[hs] (solH) at ($(solW.north west)+(-6mm,0)$) {\textbf{Projected solve}: the same, with the
      flux-balance block $D$ in the objective \eqref{eq:hs-charge}};
    \path let \p1=(solW.south), \p2=(solH.south) in
      coordinate (mrg2) at (\x1,{min(\y1,\y2)-2.5mm})
      node[box, anchor=north] (vel) at (\x1,{min(\y1,\y2)-5mm})
        {\textbf{Normal velocity}: $v_{i,j}$ at every collocation point from the motion
         law \eqref{eq:Vij}};
    \node[box, below=5mm of vel] (advW) {\textbf{One explicit step}: tangential
      redistribution $T_{i,j}$, then \eqref{eq:evolve}, the endpoints of an open curve
      held fixed};
    \node[box, below=2.5mm of advW] (tjcW) {\textbf{Triple junctions}: relocated to the
      equilibrium angles, geometry of the affected curves recomputed};
    \node[hs] (advH) at ($(advW.north west)+(-6mm,0)$) {\textbf{One constrained solve}: the
      vertices, the wall contacts and the junctions advance on their manifolds by
      \eqref{eq:hs-kkt}, so no angle is restored afterwards
      (\Section{subsec:hs-manifold})};
    \path let \p1=(tjcW.south), \p2=(advH.south) in
      coordinate (mrg3) at (\x1,{min(\y1,\y2)-2.5mm})
      node[box, anchor=north] (dis) at (\x1,{min(\y1,\y2)-5mm})
        {\textbf{Disappearance event}: a region below the threshold $\delta$ removed and
         the area rows re-baselined (\Section{subsec:disappearance})};
    \foreach \a/\b in {geo/chgW, chgW/asm, asm/solW, solW/vel, vel/advW, advW/tjcW, tjcW/dis}
      {\draw[arr] (\a) -- (\b);}
    \foreach \prev/\wh in {geo/chgW, asm/solW, vel/advW}
      {\coordinate (split-\wh) at ($(\prev.south)!0.5!(\wh.north)$);}
    \foreach \wh/\h in {chgW/chgH, solW/solH, advW/advH}
      {\draw[arr] (split-\wh) -| (\h.north)
         node[midway, above, font=\scriptsize] {half-space case};}
    \draw[arr] (chgH.south) |- (mrg1);
    \draw[arr] (solH.south) |- (mrg2);
    \draw[arr] (advH.south) |- (mrg3);
    \draw[arr] (dis.east) -- ++(0.35,0) |- ($(geo.east)+(0.35,0)$) -- (geo.east);
    \node[right,font=\footnotesize] at ($(dis.east)!0.5!(geo.east)+(0.42,0)$) {$t_{n+1}$};
  \end{tikzpicture}}
  \caption{\fix{One time step of the scheme in the whole space.
  The three branches carry the stages that the half-space construction of \Section{sec:halfspace} performs
    differently.}}
  \label{fig:algorithm}
\end{figure}
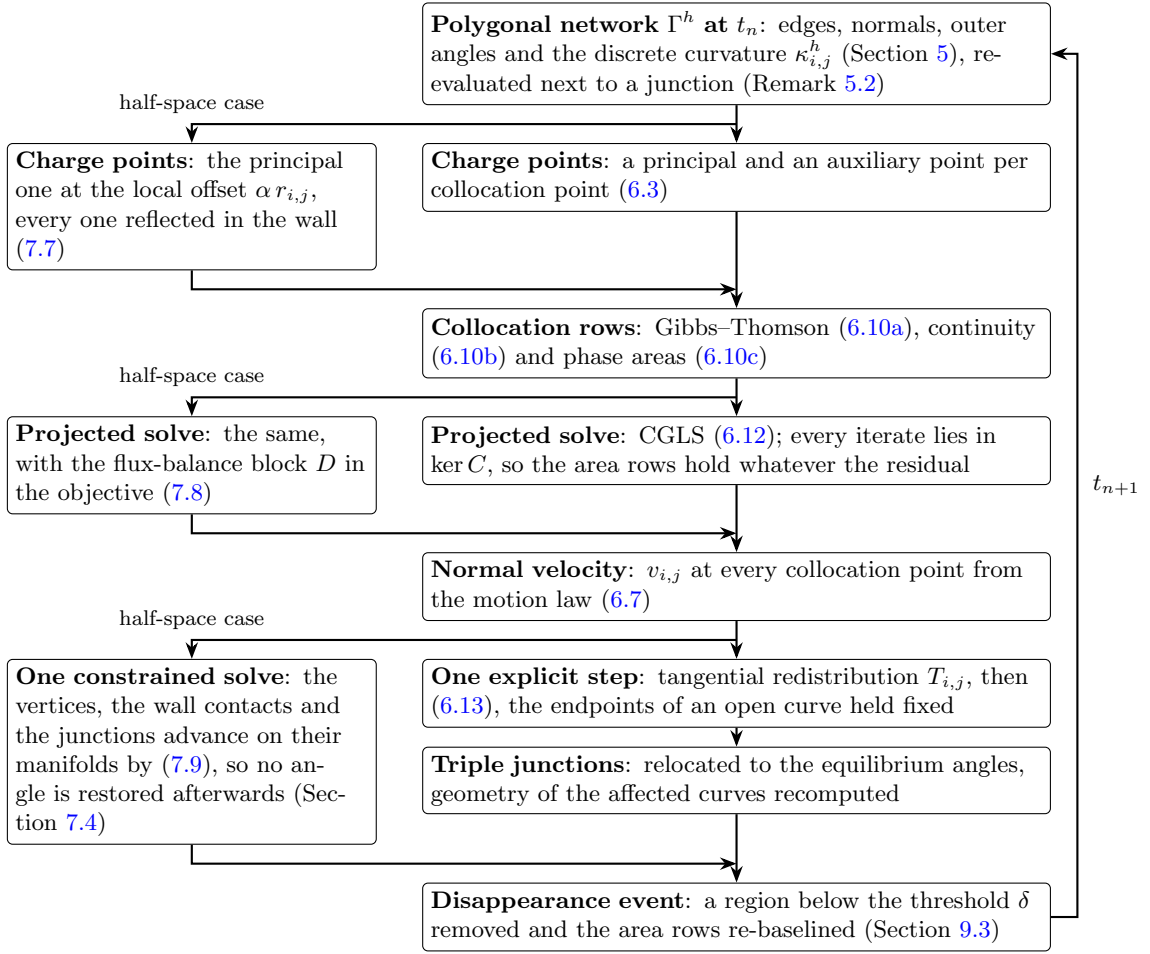

\section{Extension to Neumann boundary problem}
\label{sec:halfspace}
So far, we have explained the implementation of the fully discrete scheme in the whole space $\R{2}$.
In this section, we extend the proposed scheme to the half space $\halfspace{2}$ with boundary contacts.

\subsection{Half-space Neumann problem}
First, we introduce a target problem whose solution will be approximated by our proposed scheme.
Let $(\bv{e}_1,\bv{e}_2)$ be the standard basis of $\R{2}$.
We consider the system \eqref{system:all} in the half space $\halfspace{2}$ with the boundary wall $\wall$ defined by:
\[
    \halfspace{2} := \left\{\bv{x} = (x_1,x_2)^{\top} \in\R{2} \biggm| x_2 > 0\right\}\quad\text{and}\quad \wall := \partial\halfspace{2} = \left\{\bv{x} = (x_1,x_2)^\top\in\mathbb{R}^2 \biggm| x_2 = 0\right\}.
\]
The governing system is \eqref{system:all} with $\mathbb{R}^2$ replaced by $\halfspace{2}$, augmented by the wall condition:
\begin{equation}\label{eq:hs-neumann}
    \nabla\bvp{w}(\cdot,t)\,\bv{\nu}_{\wall} = \bvp{0}\qquad\mbox{on}\quad\wall,\quad t>0,
\end{equation}
where $\bv{\nu}_{\wall} = -\bv{e}_2$ is the outward unit normal vector to $\halfspace{2}$, which is constant in the half-space case.
An open curve may now end either at a triple junction (as in Section~\ref{sec:fully-discrete-scheme}) or at the Neumann boundary.
In the latter case, the endpoint can slide along $\wall$ (mobile) and is required to meet the boundary orthogonally,
\begin{equation}\label{eq:hs-ortho}
    \bv{\mu}_i\cdot\bv{e}_1 = 0\qquad\mbox{at a boundary contact of }\Gamma_i,
\end{equation}
i.e.,\ the interface tangent is vertical there ($90^\circ$ contact).
We stress that \eqref{eq:hs-ortho} is a prescribed neutral contact condition of the present model,
which is imposed independently of the Neumann condition \eqref{eq:hs-neumann},
which constrains the chemical potential.
The triple-junction force balance (the Herring--Young \fix{condition}, \Remark{rem:young}) is retained unchanged.

\begin{rem}
  In the two-phase case with $d = 2,\,3$, local well-posedness of the underlying problem in a smooth bounded domain with the pure Neumann boundary condition \eqref{eq:hs-neumann} and ninety-degree angle contact condition \eqref{eq:hs-ortho}
  has been established by Abels, Rauchecker, and Wilke \cite{ARW21}.
  For a stability analysis of stationary solutions to the same flow in the case $d = 2$, we refer the reader to Garcke and Rauchecker \cite{GR22}.
  We refer the reader to Hensel and Stinson \cite{HS24} for weak solutions in the case $d = 2,\,3$ with constant contact angle condition.
\end{rem}

For a source point $\bv{q} = (q_1,q_2)^\top$, let its mirror image across the boundary be
\begin{equation}\label{eq:hs-mirror}
    \bv{q}^\ast := (q_1,\; - q_2)^\top,
\end{equation}
and define the half-space (Neumann) fundamental solution by superposing the source and its image with equal sign,
\begin{equation}\label{eq:hs-image}
    \fundSolution_{\wall}(\bv{x},\bv{q}) := \fundSolution(\bv{x}-\bv{q}) + \fundSolution(\bv{x}-\bv{q}^\ast),
\end{equation}
where we recall that $\fundSolution$ is the fundamental solution to the Laplace equation in $\R{2}$.
On the boundary, the distances from $\bv{q}$ and $\bv{q}^\ast$ coincide, i.e., $|\bv{x}-\bv{q}| = |\bv{x}-\bv{q}^\ast|$ holds for $\bv{x}$ with $x_2 = 0$, and the normal derivative is
\begin{equation}\label{eq:hs-noflux}
    \frac{\partial\fundSolution_{\wall}}{\partial x_2}(\bv{x},\bv{q})
    = \frac{1}{2\pi}\left[\frac{x_2-q_2}{|\bv{x}-\bv{q}|^2} + \frac{x_2 + q_2}{|\bv{x}-\bv{q}^\ast|^2}\right]
    = 0\qquad\mbox{on}\quad x_2 = 0,
\end{equation}
since on $\wall$ the two numerators are $x_2-q_2 = -q_2$ and $x_2 + q_2 = q_2$ while the denominators are equal.
Hence, any potential built from $\fundSolution_{\wall}$ satisfies the homogeneous Neumann condition \eqref{eq:hs-neumann} on the boundary exactly,
with no additional unknowns.

\subsection{Fully discrete scheme for half-space Neumann problem}
In the CSM, the approximate solution \eqref{eq:appsol} is retained with $\fundSolution$ replaced by $\fundSolution_{\wall}$,
namely, the scalar and vector blocks remain the differences:
\begin{align}\label{eq:hs-G}
    \scalarG{r}{c_1}{\ell}{c_2}{j} &= \fundSolution_{\wall}(\vertexCenter{c_1}{\ell},\singularPoint{r}{c_2}{j}) - \fundSolution_{\wall}(\vertexCenter{c_1}{\ell},\singularPointDummy{r}{c_2}{j}),\\
    \vectorH{r}{c_1}{\ell}{c_2}{j} &= \nabla\fundSolution_{\wall}(\vertexCenter{c_1}{\ell},\singularPoint{r}{c_2}{j}) - \nabla\fundSolution_{\wall}(\vertexCenter{c_1}{\ell},\singularPointDummy{r}{c_2}{j}),\nonumber
\end{align}
where $\nabla$ is applied to the field point $\vertexCenter{c_1}{\ell}$.
The charge-point placement, however, differs from \eqref{eq:charge} as follows:
\begin{equation}\label{eq:hs-source}
    \singularPoint{r}{i}{j} := \vertexCenter{i}{j} + \alpha\,r_{i,j}\,\bv{\nu}^{\mathrm{out}}_{r,i,j},\qquad
    \singularPointDummy{r}{i}{j} := \vertexCenter{i}{j} + M_r^{\fix{\beta}}\,\bv{\nu}^{\mathrm{out}}_{r,i,j},
\end{equation}
i.e.,\ the principal source sits at a local-edge offset $\alpha\,r_{i,j}$ (with $\alpha$ an implementation parameter,
$\alpha=1.5$ in the runs) rather than at the $1/\sqrt{M_r}$ distance of \eqref{eq:charge}; the auxiliary source keeps the $M_r^{\fix{\beta}}$ far placement, and the principal-minus-dummy structure of \eqref{eq:hs-G} is retained.
The local scaling is used because the orthogonal wall contact makes the edge lengths strongly non-uniform along each curve, and anchoring the principal source to the local edge length keeps it at a fixed multiple of the local discretization scale\fix{: with $\alpha=1.5$, the principal source stands at one and a half edge lengths from the interface at every collocation point}.

\begin{rem}
    For the structure of fundamental solutions \eqref{eq:hs-image}, we have followed the previous work \cite[\S 5, Alternative fundamental solutions]{E24}
    to force the approximate solution to satisfy the pure Neumann boundary condition on $\partial\halfspace{d} = \wall$.
\end{rem}

Figure~\ref{fig:hs-schematic} illustrates this construction near the wall: the charge
points \eqref{eq:hs-source} of each region are superposed with their mirror images
$\bv{y}^{(r)\ast}_{i,j},\bv{z}^{(r)\ast}_{i,j}$ reflected across $\wall$
\eqref{eq:hs-mirror}, so that the homogeneous Neumann condition holds on $\wall$ exactly
\eqref{eq:hs-noflux}.

\begin{figure}[H]
  \centering
  \begin{tikzpicture}[>={Stealth[length=2.2mm]}, line join=round, scale=1.05]
    \definecolor{regA}{RGB}{86,180,233}   
    \definecolor{regB}{RGB}{230,159,0}    
    \coordinate (C)  at (-0.30, 0.00);
    \coordinate (V1) at (-0.30, 0.45);
    \coordinate (V2) at (-0.30, 1.35);
    \coordinate (V3) at ( 0.18, 2.42);
    \fill[regA!16] (C)--(-3.75,0)--(-3.75,3.1)--(V3)--(V2)--(V1)--cycle;
    \fill[regB!16] (C)--( 3.75,0)--( 3.75,3.1)--(V3)--(V2)--(V1)--cycle;
    \fill[regA!8] (-0.30,0)--(-3.75,0)--(-3.75,-1.90)--(-0.30,-1.90)--cycle;
    \fill[regB!8] (-0.30,0)--( 3.75,0)--( 3.75,-1.90)--(-0.30,-1.90)--cycle;
    \draw[very thick] (-3.75,0) -- (3.75,0);
    \node[below left,inner sep=3pt] at (-3.75,0) {$\wall$};
    \draw[->,black,thick] (3.30,0) -- (3.30,-0.55) node[below,inner sep=3pt] {$\bv{\nu}_{\wall}$};
    \draw[very thick] (C)--(V1)--(V2)--(V3);
    \foreach \p in {V1,V2,V3}{\fill (\p) circle (1.6pt);}
    \fill (C) circle (1.6pt);
    \node[above right,inner sep=3pt] at (V3) {$\polygon{i}$};
    \draw (C)+(0.16,0) -- +(0.16,0.16) -- +(0,0.16);
    \foreach \a/\b in {C/V1,V2/V3}{
      \coordinate (m) at ($(\a)!0.5!(\b)$);
      \draw[gray!55] (m)+(-1.6pt,-1.6pt) -- +(1.6pt,1.6pt);
      \draw[gray!55] (m)+(-1.6pt,1.6pt) -- +(1.6pt,-1.6pt);
    }
    \coordinate (M)   at ($(V1)!0.5!(V2)$);
    \coordinate (yLm) at (-1.65, 0.90);  
    \coordinate (zLm) at (-2.80, 0.90);  
    \coordinate (yRp) at ( 1.15, 0.90);  
    \coordinate (zRp) at ( 2.30, 0.90);  
    \coordinate (yLmM) at (-1.65,-0.90);
    \coordinate (zLmM) at (-2.80,-0.90);
    \coordinate (yRpM) at ( 1.15,-0.90);
    \coordinate (zRpM) at ( 2.30,-0.90);
    \foreach \a/\b in {yLm/yLmM,zLm/zLmM,yRp/yRpM,zRp/zRpM}{\draw[dashed,gray!55] (\a)--(\b);}
    \draw (M)+(-1.9pt,-1.9pt) -- +(1.9pt,1.9pt);
    \draw (M)+(-1.9pt,1.9pt) -- +(1.9pt,-1.9pt);
    \draw[->,black,thick] (M) -- ++(-0.55,0) node[above,black,inner sep=3pt] {$\vertexNormal{i}{j}$};
    \foreach \p in {yLm,zLm,yRp,zRp}{\fill[white] (\p) circle (2.0pt); \draw (\p) circle (2.0pt);}
    \foreach \p in {yLmM,zLmM,yRpM,zRpM}{\draw[gray!70] (\p) circle (2.0pt);}
    \node[above,inner sep=4pt] at (yLm) {$\singularPoint{\regionindexm{i}}{i}{j}$};
    \node[above,inner sep=4pt] at (zLm) {$\singularPointDummy{\regionindexm{i}}{i}{j}$};
    \node[above,inner sep=4pt] at (yRp) {$\singularPoint{\regionindexp{i}}{i}{j}$};
    \node[above,inner sep=4pt] at (zRp) {$\singularPointDummy{\regionindexp{i}}{i}{j}$};
    \node[below,gray!55!black,inner sep=4pt] at (yLmM) {$\bv{y}^{(\regionindexm{i})\ast}_{i,j}$};
    \node[below,gray!55!black,inner sep=4pt] at (zLmM) {$\bv{z}^{(\regionindexm{i})\ast}_{i,j}$};
    \node[below,gray!55!black,inner sep=4pt] at (yRpM) {$\bv{y}^{(\regionindexp{i})\ast}_{i,j}$};
    \node[below,gray!55!black,inner sep=4pt] at (zRpM) {$\bv{z}^{(\regionindexp{i})\ast}_{i,j}$};
    \node[right=4pt] at (V1) {$\vertex{i}{j-1}$};
    \node[right=4pt] at (V2) {$\vertex{i}{j}$};
    \node[right=4pt] at (M)  {$\vertexCenter{i}{j}$};
    \node at (-3.1, 0.36) {$\regionDiscrete{\regionindexp{i}}$};
    \node at ( 3.1, 0.36) {$\regionDiscrete{\regionindexm{i}}$};
    \node at (-3.30, 2.35) {$\halfspace{2}$};
  \end{tikzpicture}
  \caption{Half-space charge-simulation construction near the Neumann wall $\wall$.}
  \label{fig:hs-schematic}
\end{figure}
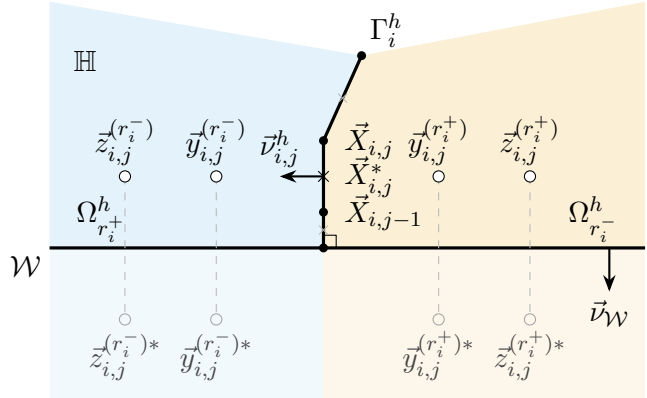

\subsection{Two distinct solves: charge field and constrained velocity}\label{subsec:hs-velocity}

\paragraph{Charge-field solve.}
The coefficients $c^{(r)}_p$ and $\charge{r}{p}{i,j}$ are obtained from the collocation system of Section~\ref{sec:fully-discrete-scheme} with the imaged blocks \eqref{eq:hs-G} and the placement \eqref{eq:hs-source}. Writing $A_1$ for the two-sided Gibbs--Thomson \eqref{eq:GTLS} and continuity \eqref{eq:CCLS} rows (right-hand side $\bv{b}_1$) and $C$ for the aggregate per-phase area rows \eqref{eq:APLS} (one per bounded phase, $I_P-1$ in total), the half-space solve reads
\begin{equation}\label{eq:hs-charge}
    \min_{\bv{q}\in\R{\unknownsNum}}\ \left\{\bigl\|A_1\bv{q}-\bv{b}_1\bigr\|_2^2 + \lambda_D\,W_D^2\,\bigl\|D\bv{q}\bigr\|_2^2\biggm|C\bv{q}=\bv{0}\right\}
\end{equation}
with the constraint enforced exactly by the orthogonal projection \eqref{eq:proj}. The block $D$ collects, per collocation edge, the flux-balance residual: the inconsistency between the two phase representations of the normal velocity \eqref{eq:Vij}. This block is the departure from the whole-space scheme, which assembles no motion-law row; it is penalized rather than enforced, with $W_D$ rescaling the rows of $D$ to the magnitude of $A_1$ so that $\lambda_D=0.1$ acts as a dimensionless relative weight keeping the Gibbs--Thomson/continuity fit dominant. From the reconstructed field the edge-center normal velocities $v_{i,j}$ are recovered as in \eqref{eq:Vij}.

\paragraph{Constrained velocity solve.}
A separate reconstruction stage takes the $v_{i,j}$ as data and produces the vertex velocities and the endpoint motion; the wall and junction constraints below are not rows of the charge-field system. Each open-curve endpoint is a triple-junction node, a wall contact ($x_2=0$), or free; let $I_W$, $I_T$, and $I_A$ count the mobile wall contacts, the trivalent junctions, and the area hard rows. The unknowns are the per-curve vertex velocities $\bv{v}_{i,j}$ and, for each triple junction $t$, a shared junction velocity $\bv{V}_t\in\R{2}$ and a scalar angular rate $\omega_t$ (every junction-incident endpoint moves with $\bv{V}_t$). The stage minimizes the fidelity/redistribution objective of \Section{subsec:hs-manifold} under the hard wall, junction, and area rows:
\begin{equation}\label{eq:hs-kkt}
\boxed{\;
\begin{aligned}
\min_{\{\bv{v}_{i,j}\},\,\{\bv{V}_t,\omega_t\}}\;
  & \left\{\textstyle\sum_i\bigl(\lVert F_i\rVert_2^2 + \lVert R_i\rVert_2^2\bigr) + \gamma\lVert G\rVert_2^2 + \lambda_{\mathrm{reg}}\lVert\bv{v}\rVert_2^2\right\}
  && \text{\small(soft rows; \S\ref{subsec:hs-manifold})}\\[-1pt]
\text{s.t.}\;
  & \bv{v}_c\cdot\bv{e}_2=0,\;\; (\bv{v}_c-\bv{v}_a)\cdot\bv{e}_1=0
  && \text{\small(wall contact, $90^\circ$)}\\
  & \tfrac{1}{\ell}(\rotation\bv{t}_{\mathrm{away}})\cdot(\bv{v}_a-\bv{V}_t)=\omega_t
  && \text{\small(junction incidence, $120^\circ$)}\\
  & \tfrac{\mathrm{d}}{\mathrm{d}t}\sum_{r\in\regionPhaseMap^{-1}(p)}\lvert\regionDiscrete{r}(t)\rvert=0
  && \text{\small(bounded region)}
\end{aligned}\;}
\end{equation}
Here $\bv{v}_c$ is a contact velocity, $\bv{v}_a$ that of the adjacent interior vertex, $\ell$ the junction-to-neighbor distance, $\bv{t}_{\mathrm{away}}$ the unit tangent pointing away from the junction, and $\rotation$ the $90^\circ$ rotation. Table~\ref{tab:hs-rows} lists the rows: the two wall rows are the velocity-level form of \eqref{eq:hs-ortho}; the shared $\omega_t$ across a junction's three incidences is an equal-angular-rate condition preserving the $120^\circ$ balance of \Remark{rem:young}; the soft families $F_i,R_i,G$ are defined in \Section{subsec:hs-manifold}; and a phase composed of several bounded regions is assigned one aggregate row $C\bv{q}=\bv{0}$ of \eqref{eq:hs-charge}.
The reconstructed $\bv{v}_{i,j}$ advance the vertices, and the junction-incident endpoints move with their $\bv{V}_t$.
The hard rows sum to
\begin{equation}\label{eq:hs-count}
    n_c = I_A + 2\,I_W + 3\,I_T ,
\end{equation}
with $I_A$ one row per constrained bounded region (or per multi-region phase in the aggregate case). In the assembled examples, each phase occupies a single region, so $I_A = I_R-1 = I_P-1$ and \eqref{eq:hs-count} gives $n_c = 18$ for the four-phase configuration of \Section{subsec:hs-validation}.

\begin{table}[tbp]
  \centering
  \footnotesize
  \setlength{\tabcolsep}{4.5pt}
  \begin{tabular}{l l l c l}
    \hline
    rows & type & definition & count & role \\
    \hline
    wall contact & hard & $\bv{v}_c\cdot\bv{e}_2=0,\ (\bv{v}_c-\bv{v}_a)\cdot\bv{e}_1=0$ & $2I_W$ & $90^\circ$ contact \eqref{eq:hs-ortho} \\
    junction incidence & hard & $\tfrac{1}{\ell}(\rotation\bv{t}_{\mathrm{away}})\cdot(\bv{v}_a-\bv{V}_t)=\omega_t$ & $3I_T$ & shared $\omega_t$ keeps $120^\circ$ \\
    area rate & hard & $\tfrac{\mathrm{d}}{\mathrm{d}t}\lvert\regionDiscrete{r}(t)\rvert=0$ & $I_A$ & phase-area conservation \\
    fidelity $F_i$ & soft, $1$ & $\widehat{\bv{n}}_{i,k}\cdot\bv{v}_{i,k}=v^n_{i,k}$ & per vertex & normal-motion fidelity \\
    compatibility $G$ & soft, $\sqrt{\gamma}$ & $\bv{n}_{i,\mathrm{adj}}\cdot\bv{V}_t=v^n_{i,\mathrm{adj}}$ & per incidence & junction--arc coupling \\
    equidistribution $R_i$ & soft, $\beta_0$ & \eqref{eq:hs-equidist} & per edge & tangential mesh regularity \\
    Tikhonov & soft, $\lambda_{\mathrm{reg}}$ & $\lambda_{\mathrm{reg}}\lVert\bv{v}\rVert_2^2$ & once & positive-definite block \\
    \hline
  \end{tabular}
  \caption{Hard and soft rows of the constrained velocity solve \eqref{eq:hs-kkt}
    ($\gamma=1$, $\beta_0=10$, $\lambda_{\mathrm{reg}}=10^{-8}$; the hard counts sum
    to $n_c$ \eqref{eq:hs-count}).}
  \label{tab:hs-rows}
\end{table}

\paragraph{Region closure at the wall.}
The area-rate row needs the region's oriented boundary polygon, assembled by an incidence-driven traversal that chains the boundary curves through their shared junction nodes and closes a wall-touching region with a wall chord between its two contacts. The validated class is that of \Section{sec:halfspace}: one contiguous wall interval, bounded by two consecutive contacts, per bounded wall-touching region.

\subsection{Tangential redistribution and on-manifold endpoint update}\label{subsec:hs-manifold}
The hard rows of \eqref{eq:hs-kkt} constrain the velocities to be tangent to the wall
and junction constraint manifolds; the endpoint update below then advances the
positions on those manifolds through explicit charts (Figure~\ref{fig:hs-charts}).
The $90^\circ$ and $120^\circ$ angles therefore hold to machine precision at every
stage of the time step, independently of the step size.

\paragraph{Soft rows: fidelity and redistribution.}
The soft rows of \eqref{eq:hs-kkt} are collected in Table~\ref{tab:hs-rows}: the fidelity
rows $F_i$ pin the normal component of each owned velocity to the reconstructed raw value
($k=2,\dots,N-1$ on an open curve, every vertex on a closed one), the compatibility rows
$G$ tie each junction velocity $\bv{V}_t$ to the incident normal velocities, the
equidistribution rows $R_i$ drive the edge rates of each curve toward uniformity,
\begin{equation}\label{eq:hs-equidist}
    \dot\ell_e - \frac{\dot L}{N_e} = \beta\left(\bar\ell - \ell_e\right),\qquad
    \bar\ell := \frac{L}{N_e},
\end{equation}
with $\ell_e$ the edge lengths, $L$ the curve length, $N_e$ the number of edges
($N-1$ open, $N$ closed), and $\beta=\beta_0\,\max_i N_i$ ($\beta_0$ dimensionless,
distinct from the exponent in \eqref{eq:charge}), and a fixed Tikhonov term makes the
owned-vertex block positive-definite without perturbing the physical velocity. All are
penalties, not hard equalities: \eqref{eq:hs-equidist} is the velocity-level, curve-wide
analogue of the uniform-distribution method of \Section{sec:fully-discrete-scheme}, and
the solve targets the raw normal velocity while selecting a tangential redistribution
under the hard rows.

\paragraph{On-manifold endpoint charts.}
This update replaces the post-hoc geometric triple-junction correction of
\Section{sec:fully-discrete-scheme}. Each trivalent junction is equipped with a chart
$(\bv{J},\psi,\{\varphi_k,\ell_k\}_{k=1}^3)$, consisting of a junction point, a shared
frame angle, and, per incident ray, a fixed angular offset and an evolving length
(Figure~\ref{fig:hs-charts}(a)); the incident endpoint is reconstructed as
\begin{equation}\label{eq:hs-tjchart}
    \bv{X}_k = \bv{J} + \ell_k\bigl(\cos(\psi+\varphi_k),\,\sin(\psi+\varphi_k)\bigr);
\end{equation}
the offsets $\varphi_k$ being fixed, the $120^\circ$ balance is exact for all $t$. Each
wall contact is equipped with a chart $(X,h)$ (Figure~\ref{fig:hs-charts}(b)),
\begin{equation}\label{eq:hs-wallchart}
    \bv{X}_{\mathrm{contact}} = (X,\,0)^\top,\qquad \bv{X}_{\mathrm{adj}} = (X,\,h)^\top,
\end{equation}
so the contact edge is vertical and the $90^\circ$ angle is exact.
A two-stage Heun update advances the chart coordinates (a predictor at $t$, a corrector
at the tentative state, the rates averaged); the constraints
\eqref{eq:hs-tjchart}--\eqref{eq:hs-wallchart} hold to machine precision at each stage.
The update is formally second order (Heun/RK2).

\begin{figure}[H]
  \centering
  \begin{tikzpicture}[>={Stealth[length=2.2mm]}, line join=round, scale=1.05]
    \coordinate (J) at (0,0);
    \draw[dashed,gray!60] (J) -- (2.0,0);
    \draw[->,gray!80] (J) -- (20:1.60);
    \draw[gray!80] (1.10,0) arc[start angle=0, end angle=20, radius=1.10];
    \node[gray!50!black,inner sep=2pt] at (10:1.42) {$\psi$};
    \coordinate (X1) at (90:1.45);
    \coordinate (X2) at (210:1.15);
    \coordinate (X3) at (330:1.80);
    \foreach \p in {X1,X2,X3}{\draw[very thick] (J) -- (\p);}
    \draw[gray!80] (20:0.55) arc[start angle=20, end angle=90, radius=0.55];
    \node[gray!50!black] at (57:0.82) {$\varphi_1$};
    \draw[gray!55] (X1) -- ++(80:0.50) -- ++(60:0.50);
    \draw[gray!55] (X2) -- ++(215:0.45) -- ++(240:0.45);
    \draw[gray!55] (X3) -- ++(320:0.50) -- ++(300:0.50);
    \fill (J) circle (1.8pt);
    \foreach \p in {X1,X2,X3}{\fill (\p) circle (1.6pt);}
    \node[above right=1pt,fill=white,inner sep=1pt] at (J) {$\bv{J}$};
    \node[left,inner sep=3pt] at (X1) {$\bv{X}_1$};
    \node[below,inner sep=3pt] at (X2) {$\bv{X}_2$};
    \node[above right,inner sep=2pt] at (X3) {$\bv{X}_3$};
    \node[left=3pt,inner sep=1pt] at (90:0.76) {$\ell_1$};
    \node[above=3pt,inner sep=1pt] at (210:0.68) {$\ell_2$};
    \node[above=3pt,inner sep=1pt] at (330:0.94) {$\ell_3$};
    \draw[->] (235:0.42) arc[start angle=235, end angle=305, radius=0.42];
    \node[below,inner sep=1pt] at (270:0.58) {$\omega_t$};
    \node at (0,-2.15) {(a)};
    \begin{scope}[shift={(5.6,-0.55)}]
      \draw[very thick] (-1.9,0) -- (2.1,0);
      \node[below,inner sep=3pt] at (-1.55,0) {$\wall$};
      \coordinate (C) at (0,0);
      \coordinate (A) at (0,1.30);
      \draw[very thick] (C) -- (A);
      \draw[gray!55] (A) -- ++(75:0.55) -- ++(55:0.55);
      \draw (0.16,0) -- (0.16,0.16) -- (0,0.16);
      \fill (A) circle (1.6pt);
      \fill[white] (C) circle (2.0pt); \draw (C) circle (2.0pt);
      \node[below,inner sep=5pt] at (C) {$(X,\,0)$};
      \node[right,inner sep=3pt] at (A) {$(X,\,h)$};
      \draw[<->] (-0.85,-0.60) -- (0.85,-0.60) node[midway,below,inner sep=3pt] {$\dot X$};
      \draw[->] (-0.35,0.70) -- (-0.35,1.30) node[midway,left,inner sep=3pt] {$\dot h$};
      \node at (0,-1.60) {(b)};
    \end{scope}
  \end{tikzpicture}
  \caption{On-manifold endpoint charts: (a) the triple-junction chart
    \eqref{eq:hs-tjchart}, with fixed ray offsets $\varphi_k$; (b) the wall-contact
    chart \eqref{eq:hs-wallchart}, with a vertical contact edge.}
  \label{fig:hs-charts}
\end{figure}
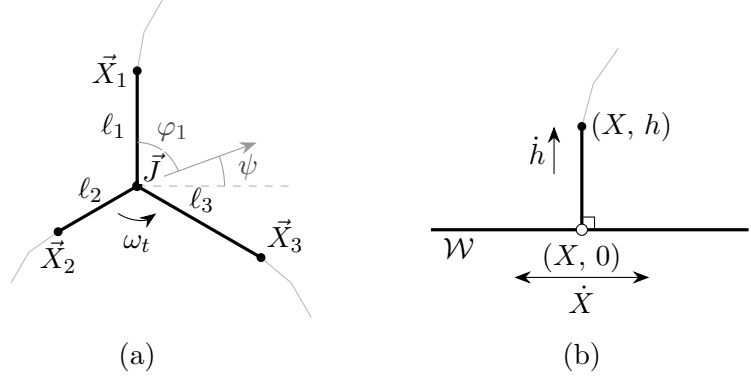

\paragraph{Adaptive time step.}
Unlike the whole-space runs of Section~\ref{sec:fully-discrete-scheme},
which use the fixed rule $\Delta t=c\,N^{-2}$, we choose the time step adaptively in the half-space runs as
\begin{equation}\label{eq:hs-dt-rule}
    \Delta t = \min\left\{c_{\mathrm{CFL}}\,\frac{\min_{i,j}\ell_{i,j}}{\max_\Gamma|V|},\ \Delta t_{\max}\right\},
\end{equation}
where $\min_{i,j}\ell_{i,j}$ is the smallest polygonal edge length in the network, $\max_\Gamma|V|$ the maximum discrete normal speed, $c_{\mathrm{CFL}}>0$ a Courant constant, and $\Delta t_{\max}$ a cap (in the runs $c_{\mathrm{CFL}}=6.25\times10^{-3}$ and $\Delta t_{\max}=10^{-3}$). The step \eqref{eq:hs-dt-rule} contracts when a high-curvature feature drives up $\max_\Gamma|V|$ and relaxes as the network smooths.

\section{Structural properties of the discrete scheme}\label{sec:analysis}

We now show that the projected least-squares solve of
Section~\ref{sec:fully-discrete-scheme} preserves, at the discrete level, the
area-conservation structure of the continuous flow (Proposition~\ref{prop:ap}). For the
$p$-th bounded phase, we define a linear functional of the unknown vector $\bv{q}$
approximating the variation of area functional:
\begin{equation}\label{eq:analysis-Jp}
    J_p(\bv{q}) := \sum_{r\in \regionPhaseMap^{-1}(p)}\ \sum_{i\in \curveRegionMap^{-1}(r)} o_{r,i}\sum_{j=1}^{\vertexNum{i}} r_{i,j}\, v_{i,j}(\bv{q})
    \approx \int_\Gamma -V_{in}\,\dH{1},
\end{equation}
where the orientation sign $o_{r,i}$ is recalled from \eqref{eq:orientation},
and $v_{i,j}(\bv{q})$ is the approximate 
velocity in \eqref{eq:Vij}.
The area rows of the linear system are $C\bv{q}=\bv{0}$ with $C$ the $(I_P-1)\times \unknownsNum$ matrix,
where $\unknownsNum$ is the total number of unknowns of the linear system \eqref{eq:num-unknowns},
whose $p$-th row coincides with $J_p$ up to an overall sign, and the projector $P$
of \eqref{eq:proj} is used to solve the constrained least-squares problem
\eqref{eq:pcgls}.

Two structural facts underlie the statements
below. First, by construction the rows of $C$ are assembled from the same kernel
evaluations and the same $\phaseindexm{i}$-representation \eqref{eq:Vij} of the
normal velocity that is later used to advance the curves. Thus the two zero constraints
are exactly equivalent; their row conventions differ only by an overall sign. Second, we
assume throughout this section the non-degeneracy condition:
\begin{equation*}
    \operatorname{rank} C = I_P - 1, \tag{ND}\label{eq:ND}
\end{equation*}
so that the matrix $CC^\top \in\R{(I_P-1)\times(I_P-1)}$ is regular;
the redundancy of a further area row is explained by Lemma~\ref{lem:redundant} below.

\begin{prop}[Solvability of the projected charge solve]\label{prop:pls}
Assume \eqref{eq:ND}. Then, the following statements hold:
\begin{enumerate}
\item[(i)] $P = I - C^\top(CC^\top)^{-1}C$ is well-defined, symmetric, and is the
      orthogonal projector onto $\ker C$;
\item[(ii)] $C\bv{q}=\bv{0}$ holds if and only if $\bv{q}=P\bv{y}$
      for some $\bv{y}$, and the constrained problem
      $$\min_{\bv{q}}\ \left\{\|A_1\bv{q}-\bv{b}_1\|_2^2 \biggm| C\bv{q}=\bv{0}\right\}$$ is equivalent to the projected
      problem \eqref{eq:pcgls}, in the sense that the optimal values coincide and
      $\bv{q}^\star$ solves the former if and only if $\bv{q}^\star = P\bv{y}^\star$ for a
      minimizer $\bv{y}^\star$ of the latter;
\item[(iii)] the set of all solutions to the constrained problem is not empty, an affine subspace of
      $\ker C$, and it contains a unique element of minimal Euclidean norm.
\end{enumerate}
\end{prop}
\begin{rem}
    \label{rem:min-solvability}
In Proposition~\ref{prop:pls}, we need no rank assumption on $A_1$. In particular, the statement covers the
overdetermined and rank-deficient Gibbs--Thomson/continuity block that occurs in
practice.
\end{rem}
\begin{proof}
(i) Under \eqref{eq:ND}, the matrix $CC^\top$ is symmetric positive definite, and hence
it is invertible, and $P$ is well-defined and symmetric.
A direct computation gives $P^2=P$
and $CP=0$, and thus $\range{P}\subseteq\ker C$; conversely $C\bv{q}=\bv{0}$ gives
$P\bv{q}=\bv{q}$, whence $\range{P}=\ker C$, and $P=P^\top$ makes it the
orthogonal projector.

(ii) The characterization of the feasible set is (i); substituting
$\bv{q}=P\bv{y}$ turns the constrained problem into the projected one, and conversely any
minimizer $\bv{y}^\star$ yields the feasible point $\bv{q}^\star=P\bv{y}^\star$ with the same
objective, while every feasible $\bv{q}$ equals $P\bv{q}$ and so cannot do better.

(iii) The projected problem is an unconstrained linear least-squares problem, whose objective is a
convex quadratic bounded below and therefore attains its infimum; the image of its
solution set under $\bv{y}\mapsto P\bv{y}$ is the constrained solution set, a nonempty affine
subspace of $\ker C$, which contains a unique element of minimal norm.
\end{proof}

The same projector argument applies to the half-space charge-field solve of
Section~\ref{sec:halfspace}: reading the objective matrix as
$[\,A_1;\ \sqrt{\lambda_D}\,W_D D\,]$ and $C$ as its aggregate per-phase area rows,
Proposition~\ref{prop:pls} still holds. The subsequent half-space velocity and endpoint
reconstruction is a separate constrained solve, whose per-region area, wall-contact,
and junction constraints are imposed there (Section~\ref{subsec:hs-velocity}), not through
this projector.

\begin{rem}\label{rem:cgls-feasible}
The implementation solves \eqref{eq:pcgls} by CGLS with zero initial guess. In exact
arithmetic the iterates $\bv{y}_k$ then lie in the Krylov subspaces generated by
$(A_1P)^\top = PA_1^\top$, hence in $\range{P}=\ker C$; therefore
$\bv{q}_k:=P\bv{y}_k=\bv{y}_k$ and $C\bv{q}_k=\bv{0}$ for every $k$, regardless of when the
iteration is stopped. This feasibility of every iterate is the property used in
Proposition~\ref{prop:velocity-level}. Thus, feasibility of the area constraints is
independent of the stopping index and of the Gibbs--Thomson/continuity residual.
\end{rem}

\begin{lem}[One area constraint is redundant]\label{lem:redundant}
The linear functional $$\R{\unknownsNum}\ni\bv{q}\longmapsto\sum_{p=1}^{I_P} J_p(\bv{q})\in\R{}$$
is the zero map. Consequently, at most $I_P-1$ of $J_1,\dots,J_{I_P}$ are linearly independent, and enforcing $J_p(\bv{q})=0$ for
$p=1,\dots,I_P-1$ implies that $J_{I_P}(\bv{q})=0$.
\end{lem}
\begin{proof}
Fix an edge $\edge{i}{j}$ for some $i\in\naturalset{I_S}$ and $j\in\naturalset{N_i}$.
The curve $\polygon{i}$ appears in the boundary of exactly its two
adjacent regions $\regionDiscrete{\regionindexm{i}}$ and $\regionDiscrete{\regionindexp{i}}$, with orientation signs
$o_{\regionindexm{i},i}=+1$ and $o_{\regionindexp{i},i}=-1$.
Since the value $v_{i,j}$ is used per edge, it is independent of the side from which the curve is viewed
\eqref{eq:Vij}; both regions contribute the same term $r_{i,j}v_{i,j}$ to
\eqref{eq:analysis-Jp} once with sign $+1$ through the phase $\regionPhaseMap(\regionindexm{i})$
and once with sign $-1$ through $\regionPhaseMap(\regionindexp{i})$. Summing
\eqref{eq:analysis-Jp} over all $p=1,\dots,I_P$ counts every region once, hence every term
$r_{i,j}v_{i,j}$ once with each sign, and the total vanishes identically.
\end{proof}

\begin{rem}
Lemma~\ref{lem:redundant} is purely combinatorial. Indeed, in its proof, we use only that $\regionPhaseMap$ assigns each region
to exactly one phase, that each curve bounds exactly two regions, and that a single
velocity representation per edge is used on both sides.
It explains the reason that the proposed scheme assembles area rows only for $p=1,\dots,I_P-1$.
In other words, a row for $p=I_P$ would make $C$ rank
deficient by construction and $(CC^\top)^{-1}$ undefined. It also shows that the
constraint for the unbounded phase is then automatically satisfied.
\end{rem}

\begin{prop}[Exact area conservation at the velocity level]\label{prop:velocity-level}
Assume \eqref{eq:ND}. Let $\bv{q}$ be any CGLS iterate of the projected problem
\eqref{eq:pcgls} with zero initial guess; in particular, $\bv{q}$ may be the computed
solution, whatever the truncation index, the stopping tolerance, or the size of the
Gibbs--Thomson/continuity residual. Then, in exact arithmetic,
\begin{equation*}
    J_p(\bv{q}) = 0 \qquad\text{for all}\quad p=1,\dots,I_P,
\end{equation*}
i.e.\ the reconstructed normal velocity field has exactly zero net flux through the
boundary of every phase, including the omitted unbounded phase.
\end{prop}
\begin{proof}
By Remark~\ref{rem:cgls-feasible}, every iterate satisfies $\bv{q}=P\bv{y}$, hence
$C\bv{q}=CP\bv{y}=\bv{0}$ since $CP=0$ (Proposition~\ref{prop:pls}(i)). By the assembly
identity $C\bv{q}=(J_1(\bv{q}),\dots,J_{I_P-1}(\bv{q}))^\top$ this is $J_p(\bv{q})=0$ for
$p=1,\dots,I_P-1$, and Lemma~\ref{lem:redundant} gives $J_{I_P}(\bv{q})=0$.
\end{proof}

\begin{rem}[Floating point]\label{rem:analysis-fp}
In floating-point arithmetic the identity holds up to the rounding of forming and applying
$P$, controlled by the condition number of the small $(I_P-1)\times(I_P-1)$ Gram matrix
$CC^\top$, which can be monitored at run time; the residual $|J_p(\bv{q})|$ is observed at
the level of $10^{-13}$.
\end{rem}

\begin{rem}[Velocity-level versus polygonal conservation]\label{rem:vel-vs-poly}
Proposition~\ref{prop:velocity-level} concerns the reconstructed velocity, not the
advanced polygon: once the vertices have moved by \eqref{eq:evolve} and the junctions have
been relocated, the shoelace phase areas are conserved only approximately, perturbed at
higher order in $h$ and $\Delta t$ by the angle-bisector update, the junction relocation,
and the finite step. This polygonal drift is the quantity $E_A$ reported in
Section~\ref{sec:experiments}: small, linear in the time step, and decreasing under
refinement.
\end{rem}

\begin{rem}[Conservation across a region disappearance]\label{rem:disappear}
The structural conservation of Proposition~\ref{prop:velocity-level} holds between
topological events: at a removal (Section~\ref{subsec:disappearance}) the configuration is
re-baselined and the projected solve again conserves every phase area, the only change
being the removal step, at which the affected phase loses exactly the residual area
($<\delta$) removed together with the deleted curve. A region of an unconstrained phase carries
no area row, so its disappearance leaves every constrained phase area unchanged.
\end{rem}

\section{Numerical experiments}\label{sec:experiments}
In this section, we carry out numerical experiments to assess the proposed scheme.
All numerical experiments reported here use equal surface tensions $\tension{i}=1$, except
for an unequal-tension check reported at the end of Section~\ref{subsec:whole-space}.
For each test case,  we monitor the following diagnostics at every time step
to measure how closely the quantities that are conserved by the continuous flow are preserved by the discretization.

Let $A_p(t)$ denote the area of the bounded phase $p$
(the sum of the areas of its regions).
Then, the \textit{relative phase-area error} of the area of the phase $p$ at time $t$ is defined by
\begin{equation}\label{eq:diag-area}
    E_{A_p}(t) := \frac{\left|A_p(t)-A_p(0)\right|}{A_p(0)}.
\end{equation}
For later use, we let $E_A(t) := \max_{1\leq p\leq I_P-1}E_{A_p}(t)$.
Moreover, we also evaluate the \textit{triple-junction angle deviation} defined by
\begin{equation}\label{eq:diag-tj}
    E_{\angle}(t) := \max_{k\in\naturalset{I_T}}\ \max_{\ell=1,2,3}\left|\theta_{\triplejunctionIndex{k}{\ell}}(t)-120\right|,
\end{equation}
where $\theta_{\triplejunctionIndex{k}{\ell}}(t)\,(\ell=1,2,3)$ denote the three angles (degrees) separating $360^\circ$ around each triple junction point $\triplejunction{k}\,(k\in\naturalset{I_T})$.
We note that for equal surface tensions, the equilibrium angle is $120^\circ$. 
We use the convention that $E_{\angle}(t) = 0$ in the case $I_T = 0$.
We also measure the \textit{discrete interfacial energy} (weighted total length)
\begin{equation}\label{eq:diag-length}
    L^h(t) := \sum_{i=1}^{I_S}\tension{i}\,\bigl|\polygon{i}(t)\bigr|
          = \sum_{i=1}^{I_S}\tension{i}\sum_{j=1}^{\vertexNum{i}} r_{i,j}
          \approx \int_\Gamma\, \tension{}\dH{1},
\end{equation}
which is non-increasing along the continuous flow (see Proposition~\ref{prop:csp}).

We finally describe the criterion that determines when each computation is stopped.
Except for the convergence test of Section~\ref{subsec:convergence}, which is integrated
to a fixed final time, each computation is run until the flow becomes
\textit{near-stationary}: with $v^{(n)}$ the maximum discrete normal speed \eqref{eq:Vij}
at step $n$, $v_\star^{(n)}$ its running peak, and
$\rho_L^{(n)} := |L^{(n)}-L^{(n-1)}|/(L^{(n)}\,\Delta t)$ the per-step relative rate of
the interfacial energy \eqref{eq:diag-length}, the computation is stopped at the first
step $n\ge n_{\min}$ for which
\begin{equation}\label{eq:stop-rule}
    v^{(n)} < \varepsilon_v\,v_\star^{(n)} \qquad\text{and}\qquad \rho_L^{(n)} < \varepsilon_L
\end{equation}
hold on $W$ consecutive steps ($\varepsilon_v = 5\times10^{-2}$, $\varepsilon_L = 10^{-2}$,
$W = n_{\min} = 50$), subject to a maximum-step and maximum-physical-time safety cap.
Several half-space configurations of Section~\ref{subsec:hs-validation} exhibit persistent
stiff high-curvature modes for which relaxation to stationarity is impractically slow;
these are reported as finite-horizon validations over a fixed horizon rather than to
\eqref{eq:stop-rule}.

Finally, we delimit the scope of what is claimed in this section. The convergence orders
quoted below are observed rates over the tested range of resolutions, and the explicit
time-step rule $\Delta t=cN^{-2}$ of the whole-space runs is an empirical choice; no
asymptotic order or stability threshold is claimed.

\subsection{Convergence against an exact three-phase solution}\label{subsec:convergence}

To quantify the accuracy of the whole-space scheme we use the exact
radially-symmetric solution of the three-phase Mullins--Sekerka flow constructed in
\cite[\S 8.1]{EGN24}. Let three concentric circles of radii $0<R_1(t)<R_2(t)<R_3(t)$,
centered at the origin, separate the plane into four regions: the inner disk
$B_{R_1}(0)$, the inner annulus $B_{R_2}(0)\setminus B_{R_1}(0)$, the outer annulus
$B_{R_3}(0)\setminus B_{R_2}(0)$, and the unbounded exterior. The three phases are assigned
as follows: phase~$1$ is the inner annulus, phase~$2$ is the outer annulus, and
phase~$3$ is the union of the inner disk and the unbounded exterior. Thus phase~$3$ is
a single phase occupying two disconnected regions, one of them unbounded, which
tests the multi-region and non-injective $\regionPhaseMap$ part of the scheme; the
two bounded phases whose area is conserved are the two annuli. The
computed radii, interface error, and energy history are shown in
Figure~\ref{fig:eto3circle-conv}.

Since the two annulus areas are conserved, the quantities
$D_2 := R_2(0)^2-R_1(0)^2$ and $D_3 := R_3(0)^2-R_2(0)^2$ are constant, and the radii
satisfy the differential-algebraic system
\begin{equation}\label{eq:eto-dae}
    R_1(t) = \sqrt{R_2(t)^2-D_2}, \quad R_3(t) = \sqrt{R_2(t)^2+D_3}, \quad
    R_2'(t) = -F\bigl(R_2(t)\bigr),
\end{equation}
with $\tension{}\equiv 1$ and
\begin{equation}\label{eq:eto-F}
    F(u) = \frac{\dfrac{1}{\sqrt{u^2-D_2}}+\dfrac{1}{u}+\dfrac{1}{\sqrt{u^2+D_3}}}
                {2\,u\,\log\!\bigl(\sqrt{u^2+D_3}/\sqrt{u^2-D_2}\bigr)} .
\end{equation}
Following \cite{EGN24}, we obtain the reference radius $R_2(t)$ to machine precision by
inverting \eqref{eq:eto-dae} through the quadrature
$t = \int_{R_2(t)}^{R_2(0)} \mathrm{d}u/F(u)$ with a root finder, so that the reference
solution is free of time-discretization error. We take the initial radii
$R_1(0)=2$, $R_2(0)=2.5$, $R_3(0)=3$ and integrate to $T=\tfrac12$, for which
\eqref{eq:eto-dae} gives $\bigl(R_1(T),R_2(T),R_3(T)\bigr)\approx(1.60,2.20,2.75)$; the
whole configuration contracts while the two annulus areas are preserved.

Each circle is discretized with $N$ vertices and advanced with the fixed step
$\Delta t=cN^{-2}$ with $c=0.5$, admissible for this smooth radial flow
(see the cost report below). We measure the
interface error
\begin{equation}\label{eq:eto-eGamma}
    e_\Gamma(t) := \max_{i}\max_{j}\bigl|\,\|\vertex{i}{j}(t)\| - R_i(t)\,\bigr|,
\end{equation}
the per-circle radius errors $e_{R_i}(t) := |\bar R_i^h(t)-R_i(t)|$ with $\bar R_i^h$ the
mean vertex radius of circle $i$, and the diagnostics $E_A$ \eqref{eq:diag-area} and
$L$ \eqref{eq:diag-length}. \fix{Since the exact interface is a circle of radius $R_i(t)$, its
curvature is known as well; with the sign convention of \Section{sec:spatial_discretization}
its discrete counterpart should equal $+1/R_i(t)$, each circle being convex toward
$\regionDiscrete{\regionindexm{i}}$. We measure that in the arclength-weighted norm}
{\color{revision}
\begin{equation}\label{eq:eto-eKappa}
    e_\varkappa(t) := \sqrt{\frac{\sum_{i=1}^{I_S}\sum_{j=1}^{N_i} r_{i,j}\left(\curvatureDisc{i}{j}(t)-\frac1{R_i(t)}\right)^2}
                               {\sum_{i=1}^{I_S}\sum_{j=1}^{N_i} r_{i,j}}}.
\end{equation}
}
Table~\ref{tab:eto-conv} reports the run parameters,
$e_\Gamma(T)$, \fix{$e_\varkappa(T)$,} and $E_A(T)$ for a sequence of refinements
(the innermost circle has the largest radius error, which essentially equals
$e_\Gamma$, while the outer radius errors are smaller), while
Figure~\ref{fig:eto3circle-conv} shows the computed radii against the
exact ones, the interface error against $N$, and the energy history.

\begin{table}[tbp]
  \centering
  \begin{tabular}{r c r c c c c c}
\hline
$N$ & $\Delta t$ & steps & $e_\Gamma(T)$ & EOC\textcolor{revision}{$_\Gamma$} & \textcolor{revision}{$e_\varkappa(T)$} & \textcolor{revision}{EOC$_\varkappa$} & $E_A(T)$ \\
\hline
16 & $1.95\times10^{-3}$ & 256 & $1.13\times10^{-1}$ & -- & \textcolor{revision}{$3.47\times10^{-2}$} & \textcolor{revision}{--} & $1.93\times10^{-4}$ \\
32 & $4.88\times10^{-4}$ & 1024 & $7.94\times10^{-2}$ & 0.50 & \textcolor{revision}{$2.00\times10^{-2}$} & \textcolor{revision}{0.79} & $4.16\times10^{-5}$ \\
64 & $1.22\times10^{-4}$ & 4096 & $4.48\times10^{-2}$ & 0.83 & \textcolor{revision}{$1.05\times10^{-2}$} & \textcolor{revision}{0.94} & $8.83\times10^{-6}$ \\
128 & $3.05\times10^{-5}$ & 16384 & $2.04\times10^{-2}$ & 1.14 & \textcolor{revision}{$4.60\times10^{-3}$} & \textcolor{revision}{1.18} & $1.95\times10^{-6}$ \\
\hline
\end{tabular}

  \caption{Spatial convergence and run parameters for the exact three-circle solution
    \eqref{eq:eto-dae} in the whole space ($c=0.5$, $\tension{}=1$, integrated to
    $T=\tfrac12$; $I_R=4$, $I_P=3$, $I_S=3$, $I_T=0$). \fix{The interface error
    \eqref{eq:eto-eGamma} and the curvature error \eqref{eq:eto-eKappa} are listed with their
    experimental orders of convergence.}}
  \label{tab:eto-conv}
\end{table}

\begin{figure}[tbp]
  \centering
  \includegraphics[width=0.8\linewidth]{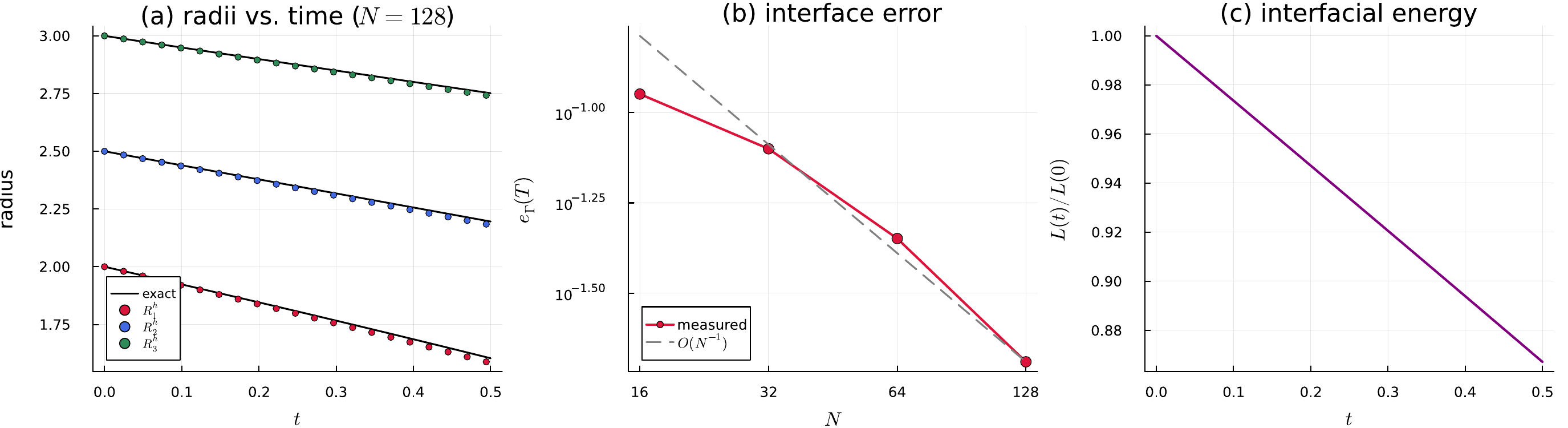}
  \caption{Convergence for the three-circle test. (a) Computed mean radii
    $\bar R_i^h(t)$ (markers) against the exact $R_i(t)$ (solid), $N=128$.
    (b) Interface error $e_\Gamma(T)$ against $N$, log--log, with an $O(N^{-1})$
    reference slope. (c) Interfacial energy $L(t)/L(0)$, $N=128$.}
  \label{fig:eto3circle-conv}
\end{figure}

The interface error decreases under refinement, with an experimental order of
convergence (EOC) that grows from about $0.5$ on the coarsest mesh to $1.1$ at $N=128$
(Table~\ref{tab:eto-conv}). \fix{The curvature error decreases at the same rate, from an EOC
of $0.8$ to $1.2$ over the same refinements, so the geometry entering the Gibbs--Thomson
condition converges no more slowly than the interface position itself.} The phase-area drift $E_A(T)$ stays below $2\times10^{-4}$ and decreases with
$N$. It is the higher-order polygonal drift of Remark~\ref{rem:vel-vs-poly}, not a failure
of conservation: the projection conserves the discrete area at the velocity level to
machine precision (Proposition~\ref{prop:velocity-level}), whereas the shoelace areas of
the advanced polygon drift by an amount that scales linearly with the step: at the
tenfold smaller $c=0.05$ it is about ten times smaller. The interfacial energy $L(t)$ decreases
monotonically, consistent with the curve-shortening property (Proposition~\ref{prop:csp}).

The observed rate approaches first order over the tested resolutions. This rate pertains
to the complete discretization, which combines the polygonal geometry, the charge-based
field reconstruction, the velocity transfer, and the time stepping; the present
experiments do not isolate the individual contributions of these components. The
reconstructed velocity error also decreases from $1.2\times10^{-1}$ at $N=64$ to
$5.5\times10^{-2}$ at $N=128$.

{\color{revision}
\paragraph{Sensitivity to the charge-point placement.}
The distance at which the charge points are held is a free parameter of any method of this
type, and this benchmark measures what it costs. Table~\ref{tab:charge-placement} replaces
the principal distance $1/\sqrt{M_r}$ of \eqref{eq:charge} by $c/\sqrt{M_r}$ at $N=32$; the
local offset $\alpha\,r_{i,j}$ of the half space \eqref{eq:hs-source} is not varied here. As
$c$ grows from one to four the interface error falls by a factor of six, while the area
error moves by less than a factor of two; past $c=4$ neither improves further, and
$\kappa_2(A_1P)$ and the CGLS iteration count do not respond anywhere in the range. The
saturation coincides with a geometric limit: at $N=32$ the distance $4/\sqrt{M_r}$ equals
$0.5$, the separation between the circles here, beyond which the charge points of one
region would lie past its neighbor's interface. That separation caps the admissible range
from above, and the runs reported in this paper keep the
conservative $c=1$ of \cite[\S 4]{E24}, which leaves that margin on every configuration
treated here.
}

\begin{table}[tbp]
  \centering
  {\color{revision}\begin{tabular}{c c c c c c}
\hline
$c$ & $\mathrm{dist}$ & $e_\Gamma(T)$ & $E_A(T)$ & $\kappa_2(A_1P)$ & CGLS its \\
\hline
0.25 & 0.031 & $5.35\times10^{-1}$ & $1.89\times10^{-4}$ & $9.6\times10^{17}$ & 9 \\
0.50 & 0.062 & $2.24\times10^{-1}$ & $7.45\times10^{-5}$ & $1.6\times10^{18}$ & 9 \\
1.00$^\ast$ & 0.125 & $7.94\times10^{-2}$ & $4.16\times10^{-5}$ & $1.4\times10^{18}$ & 9 \\
2.00 & 0.250 & $2.52\times10^{-2}$ & $3.21\times10^{-5}$ & $2.9\times10^{18}$ & 9 \\
4.00 & 0.500 & $1.35\times10^{-2}$ & $3.02\times10^{-5}$ & $2.7\times10^{18}$ & 9 \\
8.00 & 1.000 & $1.28\times10^{-2}$ & $3.01\times10^{-5}$ & $2.3\times10^{18}$ & 9 \\
\hline
\end{tabular}
}
  \caption{\fix{Sensitivity to the charge-point placement: the principal distance
    \eqref{eq:charge} is replaced by $c/\sqrt{M_r}$, with the auxiliary exponent unchanged.
    The column $\mathrm{dist}$ is the resulting distance, $^\ast$ marks the value used
    throughout the paper, and the last two columns are read at $t=0$, the error columns being
    full runs to $T=1/2$.}}
  \label{tab:charge-placement}
\end{table}

\paragraph{Cost.}
Table~\ref{tab:cost} reports the whole-space solve on this three-circle configuration at
$t=0$. The number of unknowns grows linearly in $N$; kernel assembly and the projected
solve both cost $O(N^2)$ per step (assembly dominates: $0.41$\,s versus $0.04$\,s at
$N=256$), so a run to a fixed final time costs $O(N^4)$ under $\Delta t=cN^{-2}$. No bulk
mesh is assembled. The explicit whole-space runs use the empirical rule $\Delta t=cN^{-2}$
with $c\in[0.01,0.5]$ ($c$ quoted per test); the half-space runs use the adaptive step of
Section~\ref{sec:halfspace}.

\begin{rem}[Conditioning]\label{rem:conditioning}
The projected collocation matrix is highly ill-conditioned, with $\kappa_2(A_1P)$ ranging
from $10^{17}$ to $10^{19}$ over the tested resolutions, as is
frequently observed in MFS discretizations \cite{BB08mfs,CH20}. Under the numerical
procedure used here, CGLS reaches the reported algebraic residuals in nine iterations, and
the reconstructed velocity and interface errors decrease under refinement. These
observations describe the numerical behavior of the present experiments and do not
constitute a stability analysis of the evolving-interface scheme. The area constraints, by
contrast, are applied through the small and well-conditioned Gram matrix $CC^\top$, whose
condition number remains between approximately $2.6$ and $2.8$.
\end{rem}

\begin{table}[tbp]
  \centering
  \begin{tabular}{r r c c r r}
\hline
$N$ & unknowns & CGLS its & GT resid. & asm.\ [ms] & solve [ms] \\
\hline
16 & 200 & 9 & $1.8\times10^{-10}$ & 1.11 & 0.22 \\
32 & 392 & 9 & $7.3\times10^{-9}$ & 4.77 & 0.59 \\
64 & 776 & 9 & $5.3\times10^{-8}$ & 19.22 & 1.69 \\
128 & 1544 & 9 & $9.1\times10^{-8}$ & 81.82 & 11.58 \\
256 & 3080 & 9 & $1.8\times10^{-7}$ & 407.39 & 40.04 \\
\hline
\end{tabular}

  \caption{Cost of the whole-space solve at $t=0$ for the
    three-circle configuration.}
  \label{tab:cost}
\end{table}

A limitation of this benchmark is that it contains no triple junction ($I_T=0$); no
closed-form solution with a moving triple junction is available for the multi-phase flow,
and the junction-bearing configurations are therefore assessed in
Sections~\ref{subsec:whole-space}--\ref{subsec:hs-validation}.

\subsection{Whole space}\label{subsec:whole-space}
Throughout the whole-space experiments each curve carries $N$ vertices, the time step is
$\tau = cN^{-2}$, the surface tensions are equal, and each run is stopped by the
near-stationary criterion \eqref{eq:stop-rule}. Snapshots are collected in
Figure~\ref{fig:whole-grid}, drawn on a common window per run at $t=0$, the midpoint of
the observed energy decrease, and $T$, with bounded regions colored by phase and triple
junctions marked by black dots; the area-error and energy diagnostics
\eqref{eq:diag-area}, \eqref{eq:diag-length} of the three-phase baseline are reported in
Figure~\ref{fig:whole-metrics}.

We begin with a three-phase open triple-junction baseline, which relaxes from a large
deformation to a near-stationary state with its two triple junctions preserved throughout:
the maximum phase-area drift is $E_A=3.4\times10^{-3}$, the junction angles stay within
$E_\angle=4.3\times10^{-5}$ degrees of $120^\circ$, and the interfacial energy decreases
monotonically (Figures~\ref{fig:whole-3phase} and~\ref{fig:whole-metrics}).

A four-phase network with four triple junctions probes robustness under a large geometric
distortion: the prescribed topology and the four junctions are preserved at every step, the
maximum phase-area drift is $E_A=2.9\times10^{-3}$, and the network relaxes to a
near-stationary configuration (Figure~\ref{fig:whole-4phase}).

Finally, the scheme accommodates a network with both open and closed components: an open
double-bubble coexisting with an isolated closed star-shaped curve, so the unbounded region
has two boundary components. Under the flow the star rounds toward a circle (roundness
$r_{\min}/r_{\max}\colon 0.70\to0.98$) while the open component relaxes without any topology
change; the maximum area drift is $E_A=7.1\times10^{-3}$
(Figure~\ref{fig:whole-mixed}).

\paragraph{Unequal surface tensions.}
The formulation also accommodates unequal surface tensions. As an
implementation check, we run a triple-bubble lane of three bounded regions
\fix{labeled} $A$, $B$, $C$ \fix{from left to right. Each region meets its neighbor along a
vertical interface and the unbounded region along an outer arc. The lane has} four triple
junctions, \fix{one at each end of each interior interface. We raise} the
$B$--$C$ interface tension \fix{(the right-hand one)} to $\tension{BC}=1.3$ and
\fix{leave} every other tension at one. By Remark~\ref{rem:young}, the two junctions adjacent to the $B$--$C$
interface then have the \fix{Herring--Young} angles $130.54^\circ$, $130.54^\circ$, and
$98.92^\circ$, while the two remaining junctions retain the equal-tension
value $120^\circ$. Because these angles are imposed by the geometric junction
update of Section~\ref{sec:fully-discrete-scheme} (its $120^\circ$ target
generalized to the per-junction \fix{Herring--Young} angles), this experiment checks their
preservation rather than their spontaneous emergence from the discrete
evolution\fix{; Table~\ref{tab:e1-young} lists the target angle at each junction next to the
deviation observed there}. \fix{The maximum deviation from the per-junction Herring--Young targets remains}
below $5\times10^{-5}$ degrees at both $N=32$ and $N=64$, while the maximum
phase-area drift decreases from $1.6\times10^{-3}$ to $1.3\times10^{-4}$.
\fix{The distinction matters for variational consistency. The null-space projection of
Section~\ref{sec:analysis} is assembled from the area rows alone, so the junction angles
are not an outcome of it: in the whole space, they are restored by the geometric update of
Section~\ref{sec:fully-discrete-scheme} once the vertices have moved, whereas in the half
space, the equal-angular-rate row of Table~\ref{tab:hs-rows} carries the same balance
inside the constrained velocity solve.}
The $\tension{}$-weighted interfacial energy decreases after a brief initial
adjustment, whereas the unweighted total length increases, as expected when
the surface tensions are unequal.

\begin{table}[tbp]
  \centering
  {\color{revision}\begin{tabular}{l c c c c}
\hline
junction & tensions & Herring--Young angles [$^\circ$] & \multicolumn{2}{c}{$\max\Delta\theta\,[^\circ]$} \\
 & & & $N=32$ & $N=64$ \\
\hline
top left & $(1.0,\,1.0,\,1.0)$ & $120.0/120.0/120.0$ & $3.77\times10^{-5}$ & $3.75\times10^{-5}$ \\
bottom left & $(1.0,\,1.0,\,1.0)$ & $120.0/120.0/120.0$ & $3.78\times10^{-5}$ & $3.76\times10^{-5}$ \\
top right & $(1.0,\,1.3,\,1.0)$ & $130.5/130.5/98.9$ & $4.75\times10^{-5}$ & $4.74\times10^{-5}$ \\
bottom right & $(1.0,\,1.3,\,1.0)$ & $130.5/130.5/98.9$ & $4.72\times10^{-5}$ & $4.73\times10^{-5}$ \\
\hline
\end{tabular}
}
  \caption{\fix{Unequal surface tensions ($\tension{BC}=1.3$) on the triple-bubble lane. For
    each of the four triple junctions the table lists the tensions
    $(\tension{a},\tension{b},\tension{c})$ of the three interfaces meeting there, the
    Herring--Young angles those tensions determine (\Remark{rem:young}), and the maximum
    deviation $\max\Delta\theta$ from those angles over the run, at $N=32$ and $N=64$.}}
  \label{tab:e1-young}
\end{table}

\begin{figure}[tbp]
  \centering
  \begin{subfigure}[t]{\linewidth}
    \centering
    \includegraphics[width=\linewidth]{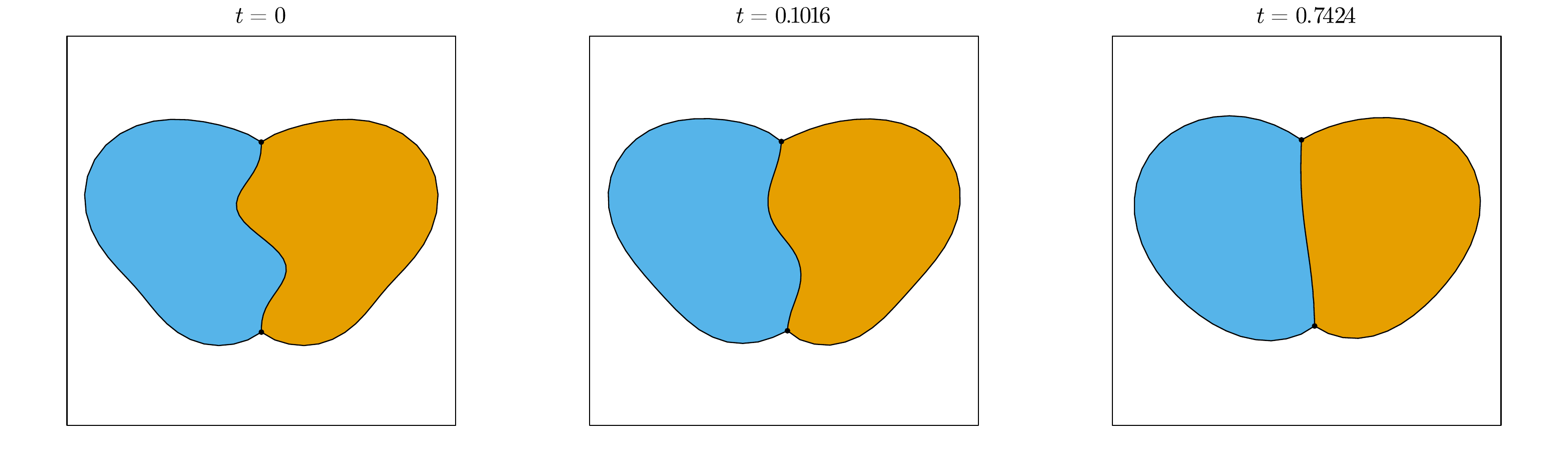}
    \caption{Three-phase baseline ($I_R=I_P=I_S=3$, $I_T=2$), $T\approx0.74$.}
    \label{fig:whole-3phase}
  \end{subfigure}\par\smallskip
  \begin{subfigure}[t]{\linewidth}
    \centering
    \includegraphics[width=\linewidth]{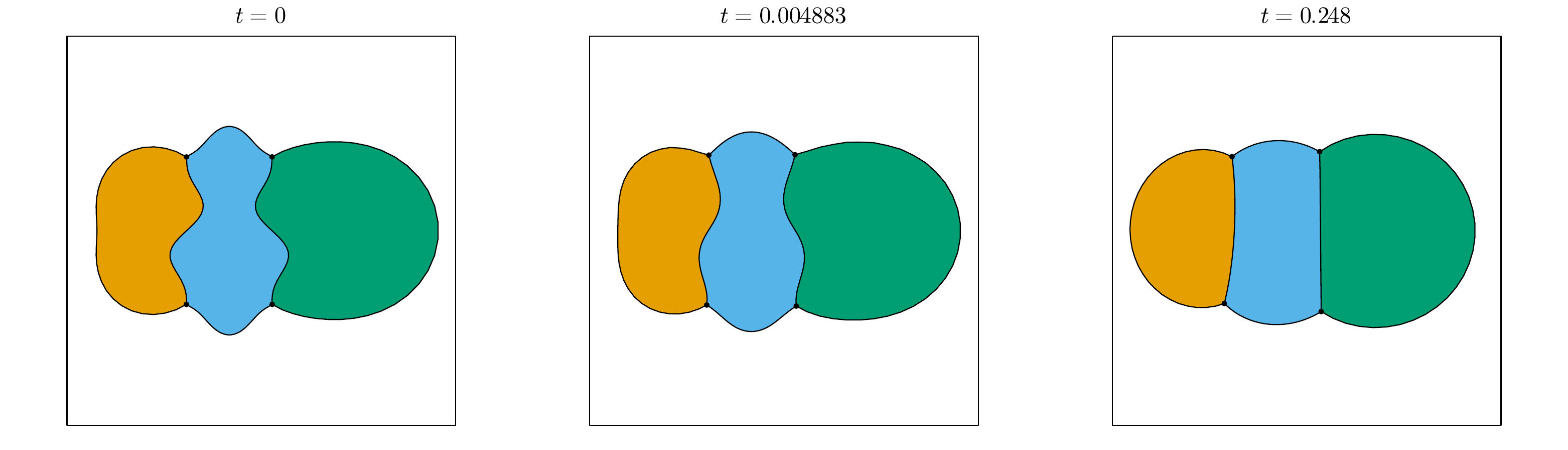}
    \caption{Four-phase fixed-topology network ($I_R=I_P=4$, $I_S=6$, $I_T=4$), $T\approx0.25$.}
    \label{fig:whole-4phase}
  \end{subfigure}\par\smallskip
  \begin{subfigure}[t]{\linewidth}
    \centering
    \includegraphics[width=\linewidth]{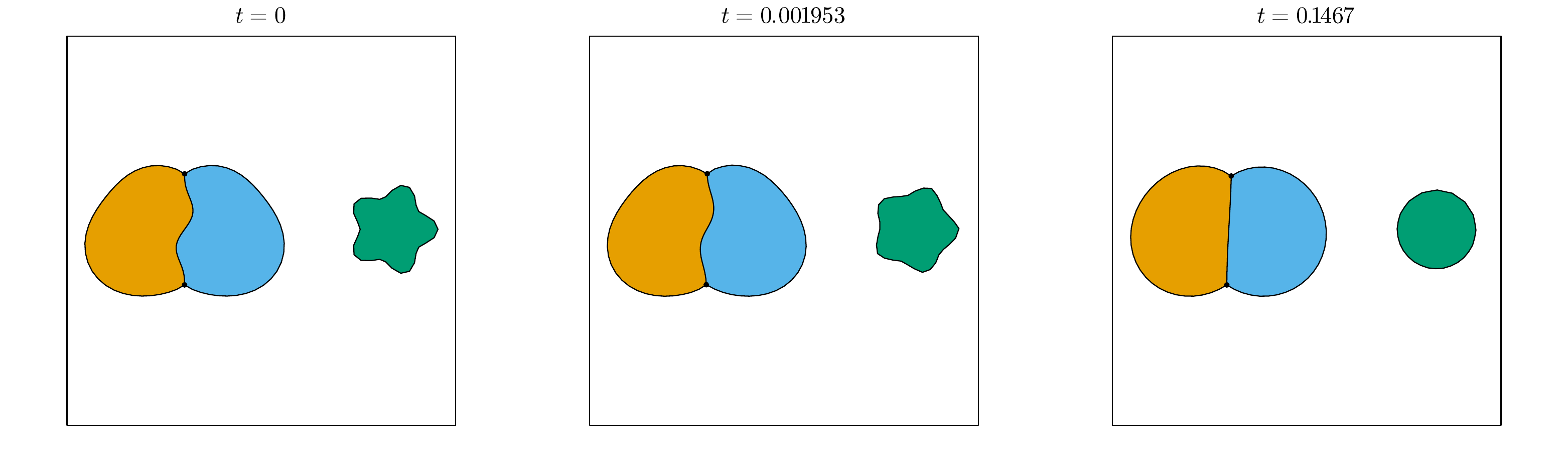}
    \caption{Mixed open/closed configuration ($I_R=I_P=4$, $I_S=4$ with one closed curve,
      $I_T=2$), $T\approx0.15$.}
    \label{fig:whole-mixed}
  \end{subfigure}
  \caption{Whole-space experiments ($N=32$, $c=0.05$, run to the near-stationary
    criterion \eqref{eq:stop-rule}): snapshots at $t=0$, the midpoint of the observed
    energy decrease, and $T$.}
  \label{fig:whole-grid}
\end{figure}

\begin{figure}[tbp]
  \centering
  \includegraphics[width=0.78\linewidth]{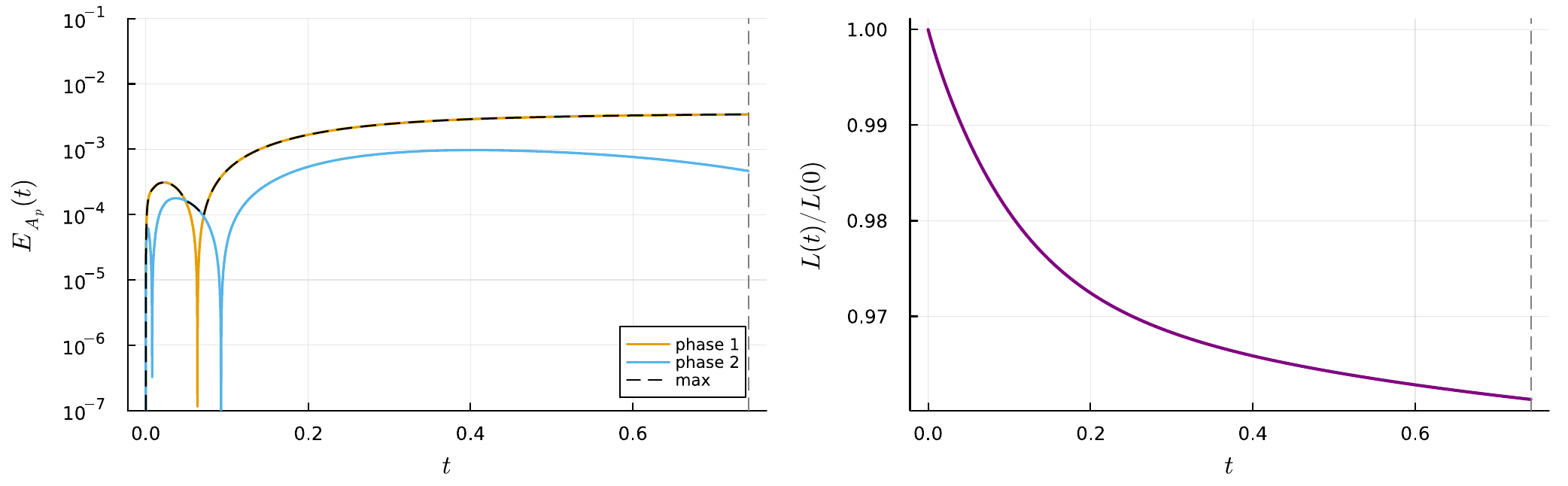}
  \caption{Area-error and energy diagnostics \eqref{eq:diag-area}, \eqref{eq:diag-length}
    for the three-phase baseline in the first row of Figure~\ref{fig:whole-grid}.}
  \label{fig:whole-metrics}
\end{figure}

\subsection{Region disappearance}\label{subsec:disappearance}
In the experiments so far, the network topology does not change.
In this subsection, we treat the disappearance of a bounded region that is enclosed by a single closed curve and
incident to no triple junction, once this curve has contracted below the resolution of the
mesh. Because the area constraints pin the total area of each phase, an isolated bounded region cannot shrink on its
own, and disappearance requires the transfer of area between two regions of the
same phase, which is possible in the case where the map $\regionPhaseMap$ is non-injective (see Section~\ref{sec:network}).

\paragraph{Detection and removal.} A closed curve $\polygon{i}$ is flagged for removal
once the area it encloses falls below a prescribed threshold,
\begin{equation}\label{eq:disappear-delta}
    A_i < \delta.
\end{equation}
For a refinement sequence, a resolution-tied choice is
$\delta_N=c_\delta h_N^2$, where
$h_N:=\max_j|\edge{i}{j}|$ is the maximum edge length of the candidate curve at
the removal step and $c_\delta>0$ is held fixed. With this choice, the area removed
at an event is bounded by $c_\delta h_N^2$ and therefore vanishes as $h_N\to0$.
At that step, the curve is deleted, and the incidence data of
Section~\ref{sec:network} are updated: the vanished region is removed from
$\curveRegionMap$ and $\regionPhaseMap$, the two regions formerly separated by
$\polygon{i}$ are merged, and the conserved-area target of the affected phase is
reset to its new polygonal value. The charge-point layout and the constraint
matrix $C$ are rebuilt for the reduced configuration, and the evolution continues. Between
events, every phase area is conserved at the velocity level
(Proposition~\ref{prop:velocity-level}); the removal step changes the affected phase
total by exactly the residual $A_i<\delta$, an identity that vanishes as $\delta\to0$
along such a refinement sequence (Remark~\ref{rem:disappear}) and is reported as a single step in the
per-phase area history.

The test case corresponds to the coarsening scenario of the introduction: two dispersed
phases, each split into three disks of distinct radii, embedded in the unbounded matrix
phase (six closed curves, no triple junctions; Figure~\ref{fig:multidrop-ripening},
$N=48$). The four satellite disks vanish in order of increasing size, and the two
surviving disks relax to the near-stationary state of \eqref{eq:stop-rule}, one round disk
per phase. Between events the constrained phase totals are flat up to a polygonal drift
below $7\times10^{-5}$, and at each event the affected phase total drops by exactly the
removed sub-resolution residual ($<\delta$, $0.47\%$--$0.53\%$ of the phase area), visible
as the staircase in Figure~\ref{fig:multidrop-ripening}.

\begin{figure}[tbp]
  \centering
  \includegraphics[width=\linewidth]{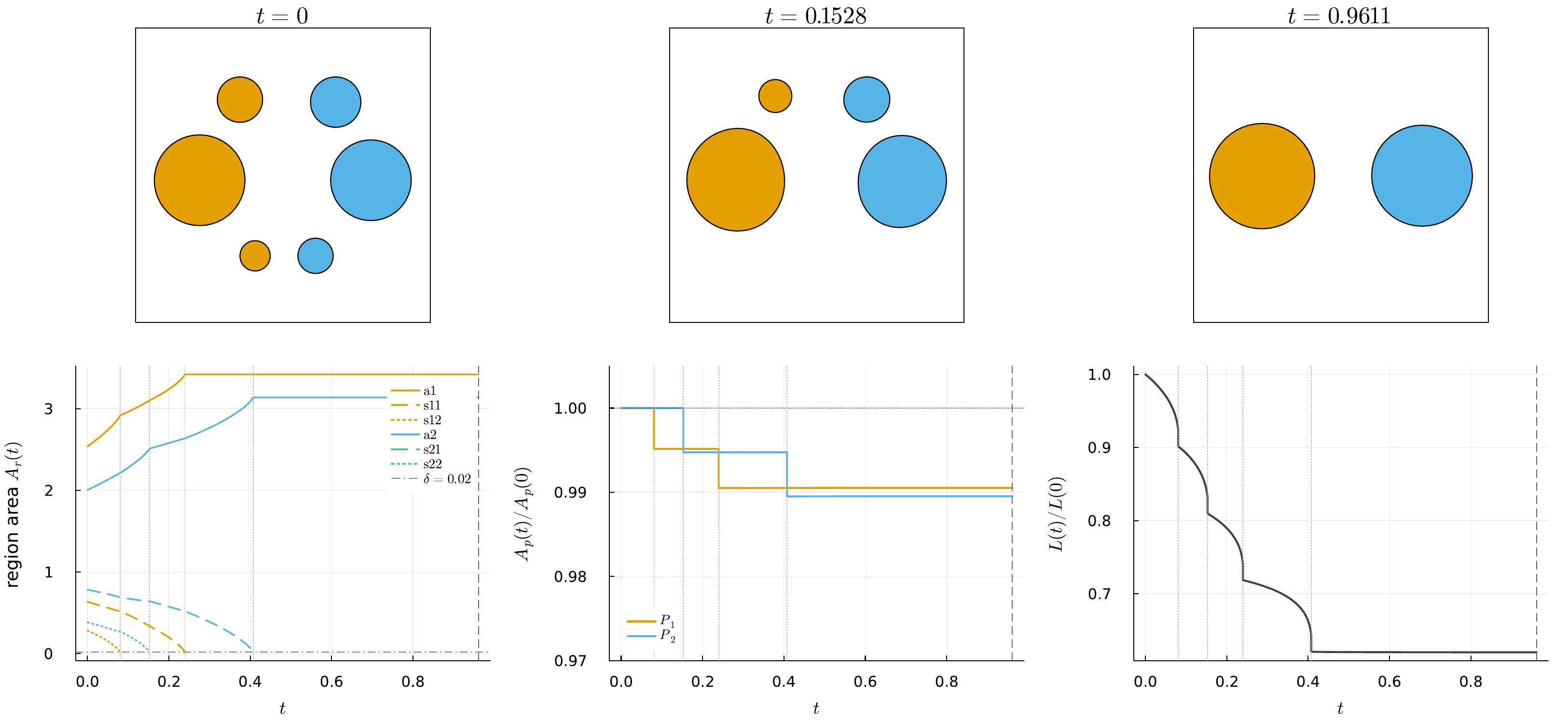}
  \caption{Intra-phase ripening of two dispersed phases in an unbounded matrix: disks of
    radii $0.90$, $0.45$, $0.30$ (phase~$1$, orange) and $0.80$, $0.50$, $0.35$
    (phase~$2$, blue), at $N=48$ per curve, $\Delta t=cN^{-2}$ with $c=0.05$, removal
    threshold $\delta=0.02$. Top: snapshots at the initial time, after the second removal,
    and at the final near-stationary state. Bottom: the six region areas $A_r(t)$ with the
    threshold $\delta$; the normalized phase totals $A_p(t)/A_p(0)$; the total interfacial
    length $L(t)/L(0)$. Dotted vertical lines mark the four removal events; the dashed
    line marks the stopping time \eqref{eq:stop-rule}.}
  \label{fig:multidrop-ripening}
\end{figure}

\begin{rem}
  Julin, Morini, Ponsiglione, and Spadaro \cite[Theorem 1.3]{JMPS23} have shown that
  any initial data $E_0\subset\R{2}$ having perimeter less than $2$ asymptotically converges to a union of disks
  whose areas sum up to the initial volume $|E_0|$ in the two-phase case under a periodic boundary condition.
  We observe a similar phenomenon in Figure~\ref{fig:multidrop-ripening} for the three-phase case.
\end{rem}

\subsection{Half space}\label{subsec:hs-validation}
We now validate the half-space scheme of Section~\ref{sec:halfspace}, in which open curves
may end on the Neumann wall as mobile contacts. All runs use the constrained velocity solve
of Section~\ref{subsec:hs-velocity}, the on-manifold Heun update of
Section~\ref{subsec:hs-manifold}, and the edge-rate redistribution constant $\beta_0=10$;
the update keeps the $90^\circ$ wall contacts and the $120^\circ$ junctions exact by
construction, the discrete area-rate rows are met to solver precision, and the polygonal
region areas drift by the small amounts quoted.
In the half-space figures the thick black line is the Neumann wall $\wall$, mobile wall
contacts are white markers, and interior triple junctions black dots; the bottom panels
report the same area-error and energy diagnostics as in Section~\ref{subsec:whole-space}.

We begin with a three-phase $M$-branch: three open curves meeting at one interior
triple junction with three mobile wall contacts, started from a strongly asymmetric,
unbalanced state ($A_2\!:\!A_1 = 7\!:\!10$). At a Gibbs--Thomson equilibrium
every arc has constant curvature, so the expected shape class is an asymmetric
circular-arc double arch. Under the flow the network relaxes toward this class: the
circle-fit residual per unit length (the root-mean-square distance of an interface from
its best-fit circle, divided by its arc length) decreases, the contact and junction angles
are held exactly by the endpoint-chart update, and the final Gibbs--Thomson residual (the
relative $\ell_2$ residual of the Gibbs--Thomson rows \eqref{eq:GTL} alone) falls below
$0.05$ (Table~\ref{tab:hs-mbranch}, Figure~\ref{fig:half-mbranch}).

\begin{figure}[tbp]
  \centering
  \includegraphics[width=0.8\linewidth]{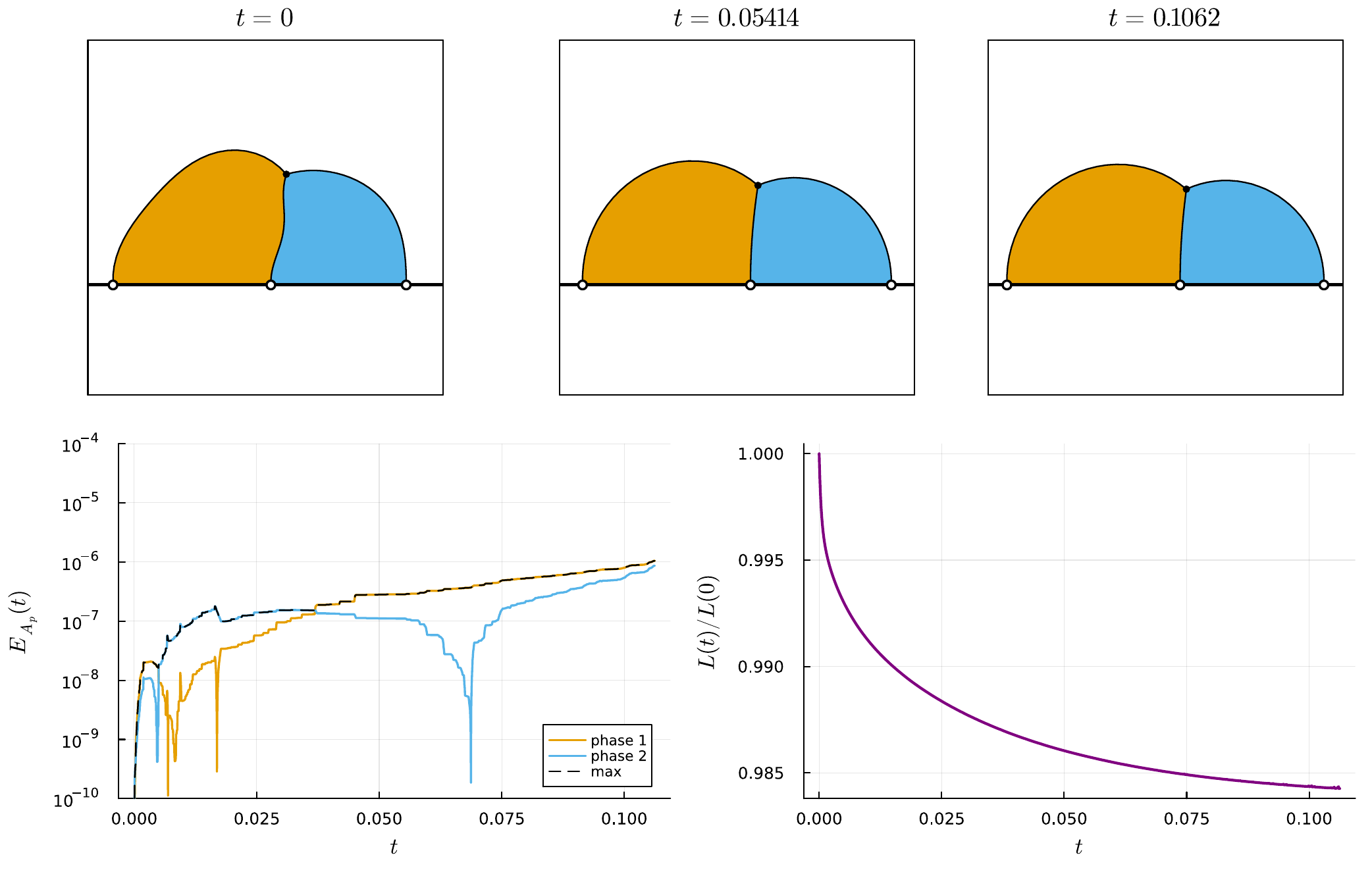}
  \caption{Half-space three-phase $M$-branch from a strongly asymmetric, unbalanced
    state ($A_2\!:\!A_1=7\!:\!10$); $N=28$, fixed horizon $T\approx0.11$. Top:
    snapshots at $t=0$, $T/2$, $T$; bottom: area-error and energy diagnostics.}
  \label{fig:half-mbranch}
\end{figure}

\begin{table}[tbp]
  \centering
  \begin{tabular}{l c}
\hline
quantity & value \\
\hline
circle-fit residual per unit length ($t=0\to T$) & $3.8\times10^{-3}$ $\to$ $1.0\times10^{-3}$ \\
max wall-contact angle deviation from $90^\circ$ & $0$ (machine) \\
max triple-junction angle deviation from $120^\circ$ & $1.4\times10^{-11}$ deg \\
max phase-area drift $E_A$ & $1.0\times10^{-6}$ \\
final Gibbs--Thomson residual & 0.019 \\
interfacial energy $L(T)/L(0)$ & 0.984 \\
\hline
\end{tabular}

  \caption{Quantitative half-space benchmark for the unbalanced $M$-branch of
    Figure~\ref{fig:half-mbranch}.}
  \label{tab:hs-mbranch}
\end{table}

To exercise many independently mobile contacts we scale the configuration up.
Figure~\ref{fig:half-4phase} shows a strongly perturbed four-phase / three-junction
network with three mobile wall contacts, relaxed toward a near-stationary asymmetric
double arch: the contacts stay ordered and orthogonal to the wall, the three junctions are
preserved, and the bounded areas are conserved to solver precision.

\begin{figure}[tbp]
  \centering
  \includegraphics[width=0.8\linewidth]{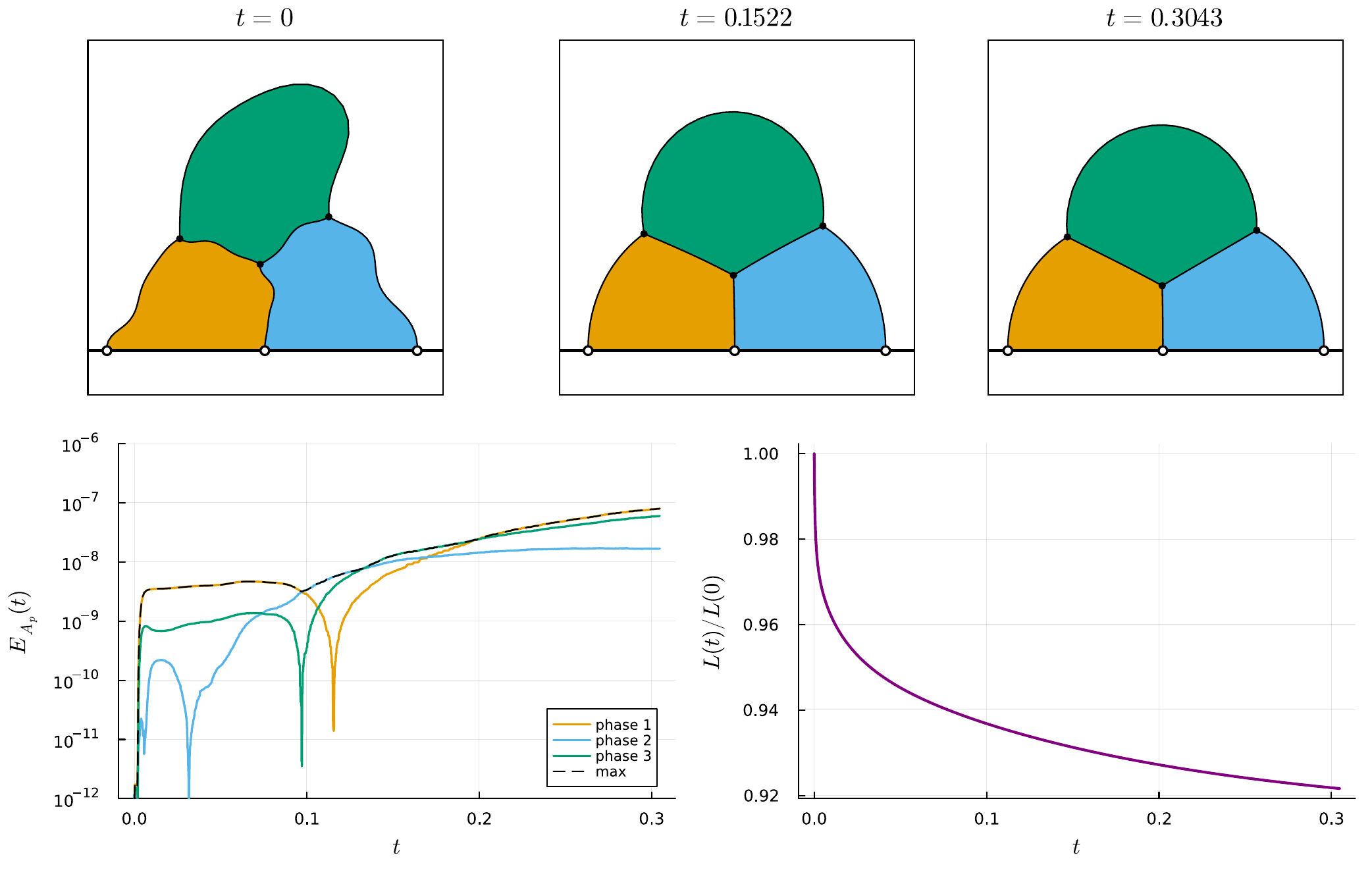}
  \caption{Half-space four-phase, three-junction network with three mobile wall
    contacts, relaxing from a strong perturbation; $N=28$, integrated to
    $T\approx0.30$. Top: snapshots at $t=0$, $T/2$, $T$; bottom: area-error and
    energy diagnostics.}
  \label{fig:half-4phase}
\end{figure}

\section{Conclusion}\label{sec:conclusion}

In this paper, we have developed a numerical scheme for the multi-phase
Mullins--Sekerka flow based on the CSM, which is a variant of the MFS, in the planar case together with the half-plane case.
Representing each chemical potential by a combination of fundamental solutions to the Laplace equation centered at
off-interface charge points has removed the need for a bulk mesh and the singular integrals,
leaving only the interfacial conditions to impose.
The discretization preserves the discrete area constraints exactly by projecting the reconstructed velocity onto their null space. Between
topological events, every bounded phase area is conserved to machine precision at the velocity level; the only topological event treated here is the
disappearance of a region enclosed by a single closed curve and incident to no junction.
The pure Neumann boundary condition is imposed exactly by image charges;
mobile contacts remain orthogonal to the boundary, and triple junctions satisfy
the Herring--Young \fix{condition}.
The accuracy of the proposed scheme has been validated against an exact three-concentric-circle solution to the underlying model.
\fix{Taken together, these give a boundary-only scheme in which no singular quadrature is required, the cost of a step is governed by the number of interface vertices alone, the wall condition holds exactly, and every bounded phase area is conserved exactly at the level of the reconstructed velocity.}

The structural result proved here is conditional on the non-degeneracy assumption
\eqref{eq:ND} and concerns exact phase-area conservation at the reconstructed-velocity
level. For curve networks with moving triple junctions and wall contacts, for which no exact
benchmark is available, the curvature correction and the half-space constrained update
have instead been assessed through refinement diagnostics, constraint residuals, area
errors, and energy dissipation. A convergence and stability analysis
of discretization of triple junctions and half-space curve networks remains an open problem.

{\color{revision}
We close by delimiting what the present work does not attempt. The scheme is planar, and
the fundamental-solution representation itself is dimension-independent.
Moreover, the interface discretization, the placement of the charge points, and the area constraints would all have
to be redesigned in three dimensions, where triple junctions become triple lines. The
surface tensions are piecewise constant and isotropic, and the unequal-tension case has
been examined only through the single check reported in
Section~\ref{subsec:whole-space}; strongly heterogeneous or anisotropic surface energies
lie outside the present validation. The Neumann boundary is restricted to a straight wall,
since the image construction is exact only for a flat boundary, and among topological events,
only the disappearance of a region enclosed by a single closed curve and incident to no junction is treated;
relaxing them so that the proposed strategy can deal with the case of three dimensions,
curved or mixed boundaries, anisotropic surface energies,
and topological transitions that involve junctions is left for future work.
}

\section*{Acknowledgments}
This research is supported by the ANR (Agence Nationale de la Recherche), grant ANR-23-PEIA-0004 (PEPR IA, PDE-AI).
\bibliographystyle{siam}
\bibliography{cite.bib}
\end{document}